\documentclass[11pt]{amsart}  
\usepackage{geometry}                % See geometry.pdf to learn the layout options. There are lots.
\usepackage{amscd}
\usepackage{graphicx}
\usepackage{amssymb}
\usepackage{epstopdf}
\newtheorem{theorem}{Theorem}

\newtheorem{definition}[theorem]{Definition}

\newtheorem{lemma}[theorem]{Lemma}

\newtheorem{proposition}[theorem]{Proposition}
\newtheorem{remark}[theorem]{Remark}

\title{{\em Local} Algebraic $K$-theory}
\author{Nicolae Teleman \\
Dipartimento di Scienze Matematiche, Universita' Politecnica delle Marche \\
E-mail:  teleman@dipmat.univpm.it}
\date{}                          % Activate to display a given date or no date

\begin{document}
\maketitle

\section{abstract}  
In this article we address the first part of the programme presented in \cite{Teleman_arXiv_III}, \S 2; we construct the {\em local} $K$-{\em theory level}  of the index formula. 
\par
Our construction is sufficiently general to encompass the algebra of pseudo-differential operators of order zero on smooth manifolds, elliptic  pseudo-differential operators of order zero,  their {\em abstract symbol} (see Introduction \S 2. ) and their {\em local} $K$-{\em theory analytical and topological index classes}, see  \cite{Teleman_arXiv_III}, \S 5, Definition 5 and 6. Our definitions are sufficiently general to apply to exact sequences of singular integral operators, which are of interest in the case of the index theorem on 
Lipschitz and quasi-conformal manifolds, see \cite{Teleman_IHES}, \cite{Teleman_Acta}, \cite{Donaldson_Sullivan}, 
\cite{Connes_Sullivan_Teleman}.
\par
In this article we introduce {\em localised algebras} (Definition 3)  $\mathit{A}$ and in \S 6 we define their {\em local} algebraic $K$-theory. 
\par
 A localised algebra $\mathit{A}$ is an algebra in which a decreasing filtration by vector sub-spaces $\mathit{A}_{\mu}$ is introduced. The filtration
 $\mathit{A}_{\mu}$ induces a filtration on the space of matrices $\mathbb{M}(\mathit{A}_{\mu})$.
\par
Although we define solely $K^{loc}_{\ast}(\mathit{A})$ for $\ast= 0, \; 1$, we expect our construction could be extended in higher degrees.  
\par
{\em We stress that our construction of} $K^{loc}_{0}(\mathit{A})$ {\em uses exclusively idempotent matrices  and that the use of finite projective modules is totally avoided. 
(Idempotent matrices, rather than projective modules, contain less arbitrariness in the description 
of the $K_{0}$ classes and allow a better filtration control).
\par
{\em The group}  $K^{loc}_{0}(\mathit{A})$ {\em is by definition the quotient space of the space of the Grothendieck completion of the space of idempotent matrices through three equivalence relations: -i) stabilisation $\sim_{s}$, -2) {\em local} conjugation} $\sim_{l}$,  {\em and} -3) {\em projective limit with respect to the filtration}.  
\par
{\em By definition, the} $K_{1}^{loc} (\mathit{A})$ {\em is the projective limit of the}  {\em local} $K_{1}(\mathit{A}_{\mu})$ {\em groups. The group}
$K_{1}(\mathit{A}_{\mu})$ {\em is by definition the quotient of } $\mathbb{GL}(\mathit{A}_{\mu})$ {\em modulo the equivalence relation
generated by: -1) stabilisation $\sim_{s}$,  --2) {\em local} conjugation $\sim_{l}$ and -3)} $\sim_{\mathbb{O}(\mathit{A}_{\mu})}$, {\em where} $\mathbb{O}(\mathit{A}_{\mu})$  {\em is the sub-module generated by elements of the form} $ u \oplus u^{-1} $,  {\em for any} $u \in \mathbb{GL}(\mathit{A}_{\mu})$.  The class of any invertible element $u$ modulo conjugation (inner auto-morphisms) we call the {\em Jordan canonical form of} $u$. 
The {\em local conjugation} preserves the {\em local Jordan canonical form 
of invertible elements}. The equivalence relation $\sim_{\mathbb{O}(\mathit{A}_{\mu})}$ insures existence of opposite elements in $K_{1}(\mathit{A}_{\mu})$ and  $K_{1}^{loc}(\mathit{A})$.
\par
{\em Our definition of  $K^{loc}_{1}(\mathit{A})$  does not use the commutator sub-group} $[\mathbb{GL}(\mathit{A}), \mathbb{GL}(\mathit{A})]$ {\em nor elementary matrices in its construction}. 
\par
We define short exact sequences of  {\em localised} algebras.  To get the corresponding (open) six terms exact sequence (Theorem 51) one has to take the tensor product of the expected six terms exact sequence by $\mathbb{Z}[\frac{1}{2}]$.   We expect the factor $\mathbb{Z}[\frac{1}{2}]$ to have important consequences.
\par
Our work shows that the basic relations (onto the stably inner-automorhism classes of invertible elements) which define $K_{1}$ reside in the {\em additive} sub-group generated by elements of the form $u \oplus u^{-1}$,   $u \in \mathbb{GL}(\mathit{A})$,  rather than in the {\em multiplicativ} commutator sub-group $[\mathbb{GL}(\mathit{A}), \mathbb{GL}(\mathit{A})]$.   
\par
Even into the case of trivially filtered algebras,  $\mathit{A}_{\mu} = \mathit{A}$,  for all $\mu \in \mathbb{N}$, the introduced group $K^{loc}_{1}(\mathit{A})$ should provide more information than the classical group $K_{1}(\mathit{A})$.
%%%%%%%%%%%%%%%%%%%%%%    Introduction   Sect . 2
\section{Introduction}    

%%%%%%%%%%%%%%%%%%%%%%%%%%%%%%%

To motivate the next considerations, we make reference to algebras of pseudo-differential operators on smooth manifolds, having in mind the index formula. The index formula is a global statement whose ingredients  may be computed by {\em local} data.  Our leading idea is to {\em localise} $K$-theory and periodic cyclic homology along the lines of the Alexander-Spanier co-homology in such a way that the new tools operate naturally with the Alexander-Spanier co-homology, see \cite{Teleman_arXiv_III}.
\par
This article deals with the first part of our programme. In this article we define {\em localised} algebras $\mathit{A}$ and we define and study their {\em local} $K$-theory, $K^{loc}_{i}(\mathit{A})$, $i = 0, 1$. The main result of this article is the (open)  six terms exact sequence Theorem 51 associated with short exact sequences of {\em local} algebras.
\par
To facilitate the understanding of this paper, we say from the very beginning that we intend to formalise the algebraic phenomena behind the following short exact sequence of operator algebras
\begin{equation}  %%%%%   Eqn  (1)  Begin
0 \longrightarrow \Psi_{-1}  \overset{\iota}{\longrightarrow}  \Psi_{0} \overset{\pi}{\longrightarrow}
\Psi_{0}/\Psi_{-1}  \longrightarrow 0,
\end{equation}  %%%%   Eqn  (1)   End
where $\Psi_{0}$ is the algebra of pseudo-differential operators whose pseudo-differential symbol is homogeneous of order $\leq 0$  and $\Psi_{-1}$ is the bi-lateral ideal of pseudo-differential operators whose symbol is homogeneous order $\leq -1$,  on the smooth manifold $M$, see \cite{Atiyah_Singer_I}, \S 5. 
\par
Any such operator $A$ has a Schwartz distributional kernel on $M\times M$. This is a singular integral operator. Cutting its distributional kernel by a smooth function $\lambda$, which has small support about the diagonal in $M\times M$ and is identically $1$ on a small neighbourhood of the diagonal, does not modify its symbol $\sigma(A) $ and modifies the original operator by a smoothing operator, i.e.  by an pseudo-differential operator of order $-\infty$; smoothing operators are contained in the ideal $\Psi_{-1}$. In other words, any pseudo-differential operator of order zero has, modulo operators of order $-\infty$, a representative with the same symbol and whose {\em support is arbitrarily small}. The image through $\pi$ of  a pseudo-differential operators $A$ into $ \Psi_{0}/\Psi_{-1}$ is called the {\em abstract symbol} of the operator $A$ and denoted by $\sigma(A)$; the passage from $A$ to its abstract symbol $\sigma(A)$ is a ring homomorphism
\begin{equation}  %%%%   Eqn   (2)   Begin
\pi: \Psi_{0}  \overset{\pi }{\longrightarrow} \Psi_{0}/\Psi_{-1}.
\end{equation}  %%%   Eqn   (2)   End
\par
We may introduce a Riemannian metric on $M$ and measure the size of the supports of pseudo-differential (integral) operators on $M$. They satisfy the property 
\begin{equation}  %%%%%   Eqn  (3)  Begin
 \mathbb{I}nt (r) \circ \mathbb{I}nt (s) \subset \mathbb{I}nt (r+s),
\end{equation}  %%%%   Eqn    {3)   End
where $\mathbb{I}nt (r)$  denotes integral operators with support in an $r$-neighbourhood of the diagonal in $M \times M$.
\par
These entitle us to speak about the support of pseudo-differential operators. The support consideration leads to the possibility to introduce a filtration on the algebra of pseudo-differential operators. 
The property (3) justifies Definition 3 of {\em localised} algebras.  
\par
As said before, in this paper we define {\em local} algebraic $K$-theory. Here, for the index formula sake, we resume ourselves to define  $K^{loc}_{i}$, $i= 0,1$, although, we  do not exclude the extension of the constructions beyond these degrees.  
\par
Next, we discuss our construction of {\em local} $K$ groups, $K^{loc}_{i}(\mathit{A})$, $i=0,1$,  for {\em localised} algebras $\mathit{A}$.  
\par 
With regard to the construction of  $K_{0}$-theory groups, we mention that one or more of the following structures are used in the literature (in the pure algebraic context,  Banach algebras or $C^{\ast}$-algebras), see  \cite{Milnor}, \cite{Karoubi},  \cite{Rosenberg}
 \cite{Rordam_Lansen_Lautsen},  \cite{Blackadar},   \cite{Weibel}: finite projective modules, idempotents and various equivalence relations: Murray-von Neumann equivalence, unitary equivalence,  continuous homotopy equivalence, elementary matrix operations.
 \par
 In the literature $K_{0}(\mathit{A})$ is built on {\em finite projective modules} $P$ over the algebra $\mathit{A}$ or {\em idempotents} 
$p$ belonging to the matrix algebra $\mathbb{M}_{n}(\mathit{A})$. The two descriptions are equivalent, but it is important to acknowledge the difference between them.
\par
Our construction of $K^{loc}_{0}$ uses exclusively idempotent matrices.  There are a few reasons why we chose to avoid projective modules: -i) projective modules, in comparison with idempotent matrices,  contain more arbitrariness in describing $K_{0}$ classes and -ii) matrices are more suitable for controlling the algebra filtration data. Matrices are more prone to make calculations.
\par
No reference to projective modules is used in our constructions. 
\par
 Regarding our construction of $K^{loc}_{1}$ we recall that  the classical algebraic $K$-theory group $K_{1}(\mathit{A})$ of the unitary algebra $\mathit{A}$, see   \cite{Whitehead}, \cite{Bass}, \cite{Milnor}, \cite{Karoubi},
 is by definition the Whitehead group  
\begin{equation*}
K_{1}(\mathit{A}) := \mathbb{GL}(\mathit{A})/ [\mathbb{GL}(\mathit{A}), \mathbb{GL}(\mathit{A})],
\end{equation*}
 where
$[\mathbb{GL}(\mathit{A}), \mathbb{GL}(\mathit{A})]$ is the commutator normal sub-group of the group of invertible matrices
$\mathbb{GL}(\mathit{A})$.
Our definition of {\em local} $K$-theory groups needs to keep track of the {\em number of multiplications} performed inside the algebra $\mathit{A}$. In order for our constructions to hold it is necessary to use constructions which use a {\em bounded}  number of multiplications. Unfortunately, in general, the number of multiplications needed to generate the whole
commutator sub-group is not bounded. It is known that the commutator sub-group is also generated by the {\em elementary} matrices.
This is the reason why our definition of $K^{loc}_{1}(\mathit{A})$  avoids entirely factorising $\mathbb{GL}(\mathit{A})$ through the commutator sub-group or the sub-group generated by elementary matrices. 
\par
$K_{0}^{loc}(\mathit{A})$, {\bf resp}. $K_{0}^{loc}(\mathit{A})$, is by definition the Grothendieck completion of the semi-group of idempotent matrices in $\mathbb{M}_{n}(\mathit{A}_{\mu})$ modulo three equivalence relations:  -i) stabilisation, -ii) local conjugation $\sim_{l}$ by invertible elements $u \in  \mathbb{GL}_{n}(\mathit{A}_{\mu})$ and -iii) projective limits with respect to $\mu \in \mathbb{N}$ (Alexander-Spanier type limit).
\par
To understand the relation between our definition of the group $K_{1}^{loc}$  and $K_{1}$, recall first the definition of  the commutator sub-group $[\mathbb{GL}(\mathit{A}, \mathbb{GL}(\mathit{A} ]$; it is generated by all multiplicative commutators
$$
[A, B] := ABA^{-1}B^{-1},  \quad for \; any \; A, B \in \mathbb{GL}_{n}(\mathit{A}).
$$ 
Supposing that $A$ and $B$ are conjugated, i.e.  $A = U B U^{-1}$, and we write  $A \sim B$,  we have
$$
A = U B U^{-1} = U B U^{-1} B^{-1} B = [ U, B ] B. 
$$ 
This shows that if $A$ and $B$ are conjugated, they differ, multiplicatively, by a commutator. To complete this remark, we say that
$A$ and $B$ are locally conjugated and write $A \sim_{l} B$ provided $A, B$ and $U$ belong to some $\mathbb{GL}_{n}(\mathit{A}_{\mu})$; here, $\mathit{A}_{\mu}$ denote the terms of the filtration of $\mathit{A}$.
\par
It is important to note that in the particular case $\mathit{A} = \mathbb{C}$ 
the quotient of $\mathbb{GL}(\mathbb{M}(\mathbb{C}))$
through the commutator sub-group  gives much less information
than $K_{1}^{loc}(\mathbb{M}(\mathbb{C}))$, which is obtained as the quotient of the same space through the action of the inner auto-morphism group; in this last case one obtains a quotient of the set of conjugacy classes, i.e.  the space of Jordan canonical forms of invertible matrices, modulo permutations of Jordan cells. 
\par
The main result of this article is the open the six terms exact sequence associated to  short exact sequence, Theorem 51. We note, however, that given the higher amount of information kept by the {\em local algebraic $K_{\ast}$ groups} we are defining here,
exactness of the six terms sequence requires tensorising the expected six terms sequence by $\mathbb{Z}[\frac{1}{2}]$. This phenomenon is related to the difficulty of lifting elements of the form $u \oplus u^{-1}$ over the quotient algebra $\mathit{A}/\mathit{A}^{'}$ to elements of the same form over the algebra $\mathit{A}$. The presence of the factor $\mathbb{Z}[\frac{1}{2}]$ should have important consequences.

%%%%%%%%%%%%%%%%%%%%%%%%%%%%%%%%%%%%%%%%%%%%

\section{Generalities and Notation.}  %%%%%%%%    Sect. 3   Generalities 
Let $\mathit{A}$ be a complex algebra, with or without unit. If the unit will be needed, the unit will be adjoined.
\par
\begin{definition}  %%%%%%%%   Definition 1        Begin  
Given the algebra $\mathit{A}$ we denote by $\mathbb{M}_{n}(\mathit{A})$ the space of $n \times n$ matrices with entries in 
$\mathit{A}$.  $\mathbb{M}_{n}(\mathit{A})$ is a bi-lateral $\mathit{A}$-module.
\par
Let  $\mathbb{I}demp_{n} \subset \mathbb{M}_{n}(\mathit{A})$ be the subset of idempotents $p$ ($p^{2} = p$) of size $n$ with entries in
$\mathit{A}$.
\par
Suppose $\mathit{A}$ has a unit. We denote by $\mathbb{GL}_{n}(\mathit{A})$ the sub-space of matrices $M$ of size $n$, with entries in $\mathit{A}$ which are invertible,  i.e. there exists the matrix $M^{-1} \in \mathbb{M}_{n}(\mathit{A})$ such that
$M M^{-1} = M^{-1} M = 1$. $\mathbb{GL}_{n}(\mathit{A})$ is a non-commutative group under multiplication.
\end{definition}     %%%%%%%   Definition  1      End
%%%%%%%%%%%%%%%%%%%%%%%%%%%%%%%%%%
\begin{definition}  %%%%%%%%   Definition 2        Begin  
The inclusions 
\par
-i) $\mathbb{I}demp_{n}(\mathit{A}) \longrightarrow \mathbb{I}demp_{n+1}(\mathit{A}) $
\begin{equation}   %%%%  Eqn ( 4 )  Begin
p \mapsto 
\begin{pmatrix}
p  & 0 \\
0  &  0
\end{pmatrix}
\end{equation}  %%%%  Eqn ( 4 )  End
\par
-ii) $\mathbb{GL}_{n}(\mathit{A}) \longrightarrow \mathbb{GL}_{n+1}(\mathit{A}) $
\begin{equation}  %%%   Eqn  (5)   Begin
M \mapsto 
\begin{pmatrix}
M  & 0 \\
0  &  1
\end{pmatrix}
\end{equation}  %%%   Eqn  (5) 
are called {\em stabilisations.}
\par
Stabilisations define two direct systems with respect to $n \in \mathbb{N}$.
\end{definition}  %%%%%%%%   Definition 2   End

%%%%%%%%%%%%%%%%%%%%%%%%%%%%%%%%%%%%%

\section{{\em Localised} Algebras}   %%%%%%  Sect. 3    Localised Algebras
\begin{definition}      %%%%%  Begin  Definition     3        (Localised algebras)  
 Localised Algebras.
\par
Let $\mathit{A}$ be an unital complex associative algebra with unit $1$.
The algebra $\mathit{A}$ is called {\em localised algebra} provided it is endowed with an additional structure satisfying the axioms (1) - (4) below.
\par  %%%%  (1)
Axiom 1. The underlying vector space $\mathit{A}$ has a decreasing filtration by vector sub-spaces
$\{ \mathit{A}_{\mu}   \}_{n \in \mathbb{N}} \subset \mathit{A}$.
\par %%%%%   (2)
Axiom 2. 
$\cup_{\mu \in \mathbb{N}} \; \mathit{A}_{\mu} \; = \mathit{A}$
\par  %%%%%  (3)
Axiom 3.
$\mathbb{C}. 1 \subset \mathit{A}_{\mu}, \; for\;  any \; \; \mu \in \mathbb{N}$
\par %%%%%%  (4)
Axiom 4.
 For any $\mu, \; \mu' \in \mathbb{N}^{+}$,   \;  $ \mathit{A}_{\mu} \; . \; \mathit{A}_{\mu'}   \subset \mathit{A}_{\mathit{Min}(\mu,\mu')-1}$ 
 ($ \mathit{A}_{0} \; . \; \mathit{A}_{0} \subset   \mathit{A}_{0}$).
\end{definition}  %%%%%%  End  Definition  3 (Localised Algebras)

\begin{remark}  %%%%%%  Begin  Remark 4
 An algebra $\mathit{A}$ could have different localisations.
\end{remark}  %%%%%  End  Remark  4
\begin{definition}    Homomorphisms of localised algebras.   Induced homomorphism. %%%%%%   Begin  Definition 5
-i) A homomorphism   from the localised algebra  $ \mathit{A} = \{ \mathit{A}_{n}   \}_{n \in \mathbb{N}}$ to the localised algebra  $ \mathit{B} = \{ \mathit{B}_{\mu}   \}_{\mu \in \mathbb{N}}$ is an algebra homomorphism  
$\phi:   \mathit{A} \longrightarrow  \mathit{B}$  such that
$ \phi: \mathit{A}_{n}    \longrightarrow \mathit{B}_{\mu}$, for any $\mu \in \mathbb{N}$. 
\par
-ii) Let $f:  \mathit{A} \longrightarrow  \mathit{B}$  be a localised algebra homomorphism.
\par
Let $f_{\ast}:  \mathbb{M}_{n}(\mathit{A}_{\mu}) \longrightarrow  \mathbb{M}_{n}(\mathit{B}_{\mu}) $ be the induced
homomorphism which replaces any component $a_{ij}$ of the matrix $M$ with the component
$f(a_{ij})$ of the matrix  $f_{\ast} (M)$. 
\end{definition}    %%%%%%   End  Definition 5
%%%%%%%%%%%%%%%%%%%%%%%%%%%%%%
\begin{remark}   %%%%%%  Begin  Remark 6
The notion of localised algebra differs from the notion of m-algebras defined by Cuntz \cite{Cuntz}, in many respects. The sub-spaces $ \mathit{A}_{n}$ are not required to be algebras, or even more, topological algebras. In the Cuntz' definition of localised Banach algebras, the projective limit of these sub-algebras might be the zero Banach algebra. However, even in these cases, the corresponding {\em local} $K$-theory could not be trivial.
\par
\end{remark}   %%%%%  End  Remark  6
\par
\begin{remark}  %%%%%   Begin Remark (7)
The immediate application of this theory regards pseudo-differential operators.
The pseudo-differential operators of non-positive order on a compact smooth manifold form a localised (Banach) algebra. The filtration  
is defined in terms of the support of the operators; the bigger the filtration order is, the smaller the supports of the operators are towards the diagonal.
\end{remark}  %%%%%   End Remark (7)

%%%%%%%%%%%%%%%%%%%%%%%%%%%%%%%%%%%%%  

\section{Mayer-Vietoris Diagrams.}   %%%%%%   Sect. 5      Mayer-Vietoris Diagrams
In this section we {\em adapt} Milnor's \cite{Milnor} construction of the first two algebraic $K$-theory groups {\em to the case of localised rings}.
\par
Let  $\Lambda, \Lambda_{1}, \Lambda_{2}$ and $\Lambda'$ be rings with unit $1$ and let
\begin{equation}   %%%%%%%   Begin Eqn  (6)
\begin{CD}
\Lambda            @>i_{1}>>   \Lambda_{1}\\
@VV{i_{2}} V                 @VV{j_{1}}V\\
\Lambda_{2}     @>j_{2}>>    \Lambda'
\end{CD}
\end{equation}   %%%%%%%%   End  Eqn  (6)
be a commutative diagram of ring homomrphisms. All ring homomorphisms $f$ are assumed to satisfy  $f(1) = 1.$ 
Any module in this paper is a left module.
\par
We assume the diagram satisfies the three conditions below.
\par
$\mathit{Hypothesis}$ 1.
All rings and ring homomorphisms are localised, see  \S 3.
\par
$\mathit{Hypothesis}$ 2. $\Lambda$ is a {\em localised product} of $\Lambda_{1}$ and $\Lambda_{2}$, i.e. for any pair of elements
$\lambda_{1} \in \Lambda_{1, \mu}$  and  $\lambda_{2} \in \Lambda_{2, \mu}$ such that $j_{1}(\lambda_{1, \mu}) = j_{2}(\lambda_{2, \mu}) = \lambda'  \in \Lambda'_{\mu}$, there exists
only one element $\lambda_{n} \in \Lambda_{\mu}$ such that  $i_{1}(\lambda_{\mu}) = \lambda_{1, \mu}$ and 
$i_{2}(\lambda_{\mu}) = \lambda_{2, \mu}$.
\par
The ring structure in $\Lambda$ is defined by 
\begin{equation}   %%%%%   Begin  Eqn  (7)
(\lambda_{1}, \lambda_{2}) + (\lambda_{1}', \lambda_{2}') := (\lambda_{1}+ \lambda_{1}' \;,\; \lambda_{2} + \lambda_{2}'), 
\hspace{0.5cm}
(\lambda_{1}, \lambda_{2}) . (\lambda_{1}', \lambda_{2}') := (\lambda_{1} . \lambda_{1}' \;,\; \lambda_{2} . \lambda_{2}'), 
\end{equation}  %%%%%%   End  Eqn  (7)
i.e. the ring operations in $\Lambda$ are performed component-wise. In \S    and \S we will need to operate with matrices with entries in
$\Lambda$. 
\par
$\mathit{Hypothesis}$ 3. At least one of the homomorphisms $j_{1}$  and $j_{2}$ is surjective.
\par
$\mathit{Hypothesis}$ 4. ({\em Optional})  For applications to the Index Theorem, see \cite{Teleman_arXiv_IV}, as the symbol is already local, we are allowed to assume, if needed,  that $ \Lambda^{'}$ is {\em localised trivially}, i.e. that  $\Lambda^{'}_{\mu} = \Lambda^{'}$ for any $\mu \in \mathbb{N}$. Hypothesis 4 is not used in this article.
\par
\begin{remark}   %%%%%   Remark  8   Begin
-i) Any matrix $M \in \mathbb{M}_{n}(\Lambda)$ consists of a pair of matrices 
$(M_{1}, M_{2}) \in \mathbb{M}_{n}(\Lambda_{1}) \times \mathbb{M}_{n}(\Lambda_{2}) $ subject to the condition
$j_{1, \ast} M_{1} = j_{2, \ast} M_{2}$.   
Any matrix $M \in \mathbb{M}_{n}(\Lambda)$ is called {\em double} matrix.
 \par
 -ii) if  $(M_{1},  M_{2})$,   $(N_{1},  N_{2})$ are double matrices, ({\bf resp.} belong to $\mathbb{M}(\Lambda_{1})  \times 
 \mathbb{M}(\Lambda_{2})$)
 and $(\lambda_{1}, \lambda_{2}) \in   \Lambda$, ({\bf resp.}  $(\lambda_{1}, \lambda_{2})  \in  \Lambda_{1} \times \Lambda_{2}$ )  
 then  relations (7) induce onto the space of double matrices, ({\bf resp.} the space $\mathbb{M}(\Lambda_{1})  \times 
 \mathbb{M}(\Lambda_{2})$)
 the following relations
 $$
(\lambda_{1}, \lambda_{2}) \; (M_{1},  M_{2}) =   (\lambda_{1} \; M_{1},  \lambda_{2} \; M_{2})
 $$
 $$
(M_{1},  M_{2}) \; +  \; (N_{1},  N_{2})  =   (M_{1} + N_{1}, \; M_{2} + N_{2})
 $$
$$
(M_{1},  M_{2}) \; .  \; (N_{1},  N_{2})  =   (M_{1} \; .  \;N_{1}, \; M_{2} \;.  \;N_{2})
 $$
 \end{remark}   %%%%%   Remark  8   End
%%%%%%%%%%%%%%
\begin{definition} %%%%%%   Begin  Definition  9
A commutative diagram satisfying Hypotheses 1. 2. 3. will be called {\em localised Mayer-Vietoris diagram}.
\end{definition}   %%%%%   End  Definition  9
%%%%%%%%%%%%%%%%%%%%%%%%%%%%%%%%%%%%%%%%%%%%%%%%%%%%%%%%%%%%

\section{$K_{0}^{loc}(\mathit{A}_{\mu})$ and  $K_{0}^{loc}(\mathit{A})$ .} %%{ Localised $K_{0}^{loc}$.  }  Sect.  6  
%%%%%%%%%%%%%%%%%%%%%%%%%%%%%%%%
\begin{definition}  %%%%%   Definition  10  Begin    \mathit{Idemp}_{n}(\mathit{A}_{\mu})
 We assume the algebra  $\mathit{A}$ is localised, see \S 4.
 \par
We consider the space of  {\em matrices with entries in}  $\mathit{A}_{\mu}$ and we  denote it by
 $\mathbb{M}_{n}(\mathit{A}_{\mu})$.
\par
Let $ \mathbb{I}demp_{n}(\mathit{A}_{\mu}) $  
denote the space of idempotent matrices of size $n$ with entries in $\mathit{A}_{\mu}$.
\par
Let $ \mathbb{GL}_{n}(\mathit{A}_{\mu}) $  
denote the space of invertible matrices $M$ of size $n$ with the property that the entries of both $M$ and $M^{-1}$ belong to
$\mathit{A}_{\mu}$. 
\par 
Let  $\mathbb{I}demp  (\mathit{A}_{\mu}) := \injlim_{n} \mathbb{I}demp_{n} (\mathit{A}_{\mu})$. 
\par
Let $\mathbb{GL}  (\mathit{A}_{\mu}) := \injlim_{n} \mathbb{GL}_{n} (\mathit{A}_{\mu})$. 
\end{definition}   %%%%%%   Definition 10   End
\par 

\begin{definition}  %%%%%%   Definition 11   Begin
-i) Two matrices $s, \; t \in \mathbb{M}_{n}(\mathit{A})$ will be called {\em conjugated} and we write $s \sim t$ provided there exists $u \in \mathbb{GL}_{n}(\mathit{A})$ such that  $s = u t u^{-1}$.
\par
-ii) Two matrices $s, \; t \in \mathbb{M}_{n}(\mathit{A}_{\mu})$ will be called {\em locally conjugated} and we write $s \sim_{l} t$ provided there exists $u \in \mathbb{GL}_{n}(\mathit{A}_{\mu})$ such that  $s = u t u^{-1}$.
\par
In particular, 
\par
-ii.1)
two idempotents $p, \; q \; \in \mathbb{I}demp_{n}(\mathit{A}_{\mu})$ are  {\em locally isomorphic} and we write $p \sim_{l} q$
 provided there exists $ u \in \mathbb{GL}_{n}(\mathit{A}_{\mu})$ such that   $q = u \; p \; u^{-1}$ and
 \par
 -ii.2)
 two invertible matrices $s, \; t \in \mathbb{GL}_{n}(\mathit{A}_{\mu})$ are {\em locally conjugated} and we write $s \sim_{l} t $ provided there exists $u \in \mathbb{GL}_{n}(\mathit{A}_{\mu})$ such that  $s = u t u^{-1}$.
\end{definition}  %%%%%%   Definition 11    End

\begin{proposition}   %%%%   Proposition  12  Begin
 $\mathit{Idemp}_{n}(\mathit{A}_{\mu})$ and  $\mathbb{GL}_{n}(\mathit{A}_{\mu})$
 are semigroups with respect to the direct sum
 \begin{gather}
 A + B :=
 \begin{pmatrix}
 A & 0\\
 0 & B
 \end{pmatrix}
 \end{gather}
 \end{proposition}  %%%%   Proposition  12    End
 %%%%%%%%%%%%%%%%%%%%%%%%%%%%%%%%
\par
 The spaces  $\mathbb{I}demp_{n}(\mathit{A}_{\mu}) $,  $\mathbb{GL}_{n} (\mathit{A}_{\mu}) $ are compatible with  stabilisations.
 \par
 The direct sum addition of idempotents, resp. invertibles,  is compatible with the local conjugation equivalence relation. Indeed, if
 $s_{1}, {s_{2}}$ are conjugated through an inner auto-morphism defined by the element $u_{1}$  ($s_{1} \sim_{l} {s_{2}}$) and
 $t_{1}, {t_{2}}$ are conjugated through an inner auto-morphism defined by the element $u_{2}$  ($t_{1} \sim_{l} {t_{2}}$), then 
 $(s_{1} + t_{1})   \sim_{l} (s_{2} + t_{2})$  are conjugated through the inner auto-morphism  $u_{1} \oplus u_{2}$. 
  \par
  With this observation, the associativity of the addition is now immediate.
  \par
These show that
  $ \mathbb{I}demp(\mathit{A}_{\mu}) \; / \sim_{l}$, resp.  $ \mathbb{GL}(\mathit{A}_{\mu}) \;/ \sim_{l}$,
 is an associative semi-group.
 %%%%%%%%%%%%%%%%%%%%%%%%%%%%%%%%
\begin{proposition}   %%%%%%   Proposition  13  Begin
The semi-groups 
 $ \mathbb{I}demp(\mathit{A}_{\mu}) \; / \sim_{l}$,   $ \mathbb{GL}(\mathit{A}_{\mu}) \;/ \sim_{l}$ are commutative.
\end{proposition}  %%%%%%   Proposition  13    End
%%%%%%%
\begin{proof}
The result follows from the following identity valid for any two matrices
$A, B \in  \mathbb{M}_{n}( \mathit{A}_{\mu}) $ 
\begin{equation}   %%%%%  Begin  Eqn  (9)
%%%%%%%%%%
\begin{pmatrix}
A & 0\\
0 & B
\end{pmatrix}
=
%%%%%%%%%%
\begin{pmatrix}
0 & -1\\
1 & 0
\end{pmatrix}
%%%%%%%%%%
\begin{pmatrix}
B & 0\\
0 & A
\end{pmatrix}
%%%%%%%%%%
\begin{pmatrix}
0 & 1\\
-1 & 0
\end{pmatrix} =
%%%%%%%%%%
%%%%%%%%%%
\begin{pmatrix}
0 & -1\\
1 & 0
\end{pmatrix}
%%%%%%%%%%
\begin{pmatrix}
B & 0\\
0 & A
\end{pmatrix}
%%%%%%%%%%
\begin{pmatrix}
0 & -1\\
1 & 0
\end{pmatrix} ^{-1}
%%%%%%%%%%
,
\end{equation}   %%%%%%%%   End  Eqn  (9)
which tells that  $(A + B) \sim_{l} (B+A)$.
\end{proof}\par
For more information about the relationship between the classical algebraic $K$-theory and the {\em local} $K$-theory,  see \S   9.
 %%%%%%%%%%%%%%%%%%%%%%%%%%%%%%%%%%%%%% 
\subsection{$K_{0}(\mathit{A}_{\mu})$ and $K_{0}^{loc}(\mathit{A})$.}  %%%%%%%  Subsection 6.1  K_{}
 \begin{definition}   %%%%%%  Definition  14  Begin     K_{0}^{loc}(\mathit(A)_{\mu}),      K_{1}^{loc}(\mathit(A)),
 Suppose $\mathit{A}$ is a localised unital associative algebra. We define
 \par
 \begin{equation}  %%%%%   Eqn  (10)      Begin    K_{0} (\mathit{A}_{\mu})
 K_{0} (\mathit{A}_{\mu}) \; = \; \injlim_{n \in \mathbb{N}} \;
Image: 
 \mathbb{G} \; (\mathbb{I}demp_{n}(\mathit{A}_{\mu}) \;/ \sim_{l}) \longrightarrow
  \mathbb{G} \; (\mathbb{I}demp_{n}(\mathit{A}_{\mu +2 }) \;/ \sim_{l})  ,
 \end{equation}    %%%%%   Eqn  (10)   End
 where $\mathbb{G}$ means Grothendieck completion, and
 \par
 \begin{equation}   %%%%%   Eqn  (11)   Begin     K_{1} (\mathit{A})
K_{0}^{loc} (\mathit{A})  \; = \; \projlim_{\mu \in \mathbb{N}}  \;  K_{0} (\mathit{A}_{\mu}).
\end{equation} %%%%%   Eqn  (11)   End
\end{definition}
The loss of two filtration orders on the formula (10) are needed to insure that $\sim_{l}$ behalves transitively in the co-domain.
\par
Any local inner automorphism induces the identity on $K_{0}^{loc} (\mathit{A})$.
%%%%%%%%%%%%%%%%%%%%%%%
\subsection{ $K_{1}(\mathit{A}_{\mu})$  and  $K_{1}^{loc} (\mathit{A})$.}  %%%%   Sect.  6.2   $K_{1}^{loc} (\mathit{A})$
\hspace{15cm}

In this sub-section $[u]_{\sim_{sl}}$ denotes the class of the element $u \in \mathbb{GL}(\mathit{A}_{\mu})$ with respect to the equivalence relation generated by stabilisation and local conjugation $\sim_{l}$.
\par
We call $[u]_{\sim_{sl}}$  the {\em abstract Jordan canonical form} of the invertible element $u$. The
$K_{1}(\mathit{A}_{\mu})$ we are going to define preserves the information provided by the abstract Jordan form. The classical definition of $K_{1}$ extracts a minimal part of the abstract Jordan form.  As the addition in the semi-group
$\mathbb{GL}(\mathit{A})$ is given by  direct sum and the Jordan canonical form $J(u)$ (at least in the classical case of the algebra $\mathbb{GL}_{n}(\mathbb{R})$) behaves additively (\; $J(u\oplus v) = J(u) \oplus J(V)$\; modulo permutations of the Jordan blocks\; ), given an arbitrary element  $u \in \mathbb{GL}(\mathit{A})$, it is not reasonable to expect existence of an element $\tilde{u}$ such that
$[u + \tilde{u}]_{\sim_{sl}} = [ 1_{2n}]_{\sim_{sl}}$. Given that we want to define $K_{1}(\mathit{A}_{\mu})$ as a group, we introduce the group structure (opposite elements) forcibly. In the case of the classical $K_{1}$, the class of the element $u^{-1}$ represents the opposite class,  $- [u] \in K_{1}(\mathit{A})$.
We intend produce a theory in which this {\em relation} is preserved. 
\par
Note that if $u_{1} \sim_{sl} u_{2}$, then  $u^{-1}_{1} \sim_{sl} u^{-1}_{2}$.
As said above, we would like the class of the element $u^{-1}$ to represent the opposite class, i.e. to think of
$u + u^{-1} \sim_{sl} I_{2n} $  as representing the zero class.
This need, along with other related facts, see Theorem 40.- i), proof of Theorem 27 and Theorem 35, lead us to concentrate our attention onto the space of matrices of the form $ u + u^{-1}$ and based on it to introduce a new equivalence relation.
%%%%%%%%%%%%%%%%%%%%%%%%%%%%% 
\begin{definition}  %%%%%   Definition  15   Begin
We define, for any $u \in \mathbb{GL}_{n}(\mathit{A}_{\mu})$
\begin{equation}  %%%%%   Eqn  (12)  Begin
\mathcal{O}(u) :=
\begin{pmatrix}
u & 0\\
0 & u^{-1}
\end{pmatrix}
\in \mathbb{GL}_{2n}(\mathit{A}_{\mu})
.
\end{equation}  %%%%%   Eqn  (12) End
\end{definition}   %%%%%   Definition  15   End
%%%%%%%%%%%
\begin{definition}   %%%%%    Definition  16   Begin
Let
\begin{equation}   %%%%%   Eqn   (13)   Begin
\mathbb{O}_{2n}(\mathit{A}_{\mu}) := \{ [\mathcal{O}(u)] \;|\; u \in   \mathbb{GL}_{n}(\mathit{A}_{\mu}) \}
\end{equation}   %%%%%%  Eqn   (13)   End
and
\begin{equation}  %%%%%   Eqn   (14)   Begin
\mathbb{O}(\mathit{A}_{\mu}) := \injlim_{\mu \in \mathbb{N}}   \mathbb{O}_{2n}(\mathit{A}_{\mu})
\{ [\mathcal{O}(u)] \;|\; u \in   \mathbb{GL}_{n}(\mathit{A}_{\mu}) \}
\end{equation}  %%%%%   Eqn   (14)
\end{definition}  %%%%%   Definition  16   End
Next we analyse the properties of this space.
%%%%%%%%
\begin{proposition}   %%%%%%  Proposition   17   Begin
-i) The space $\mathbb{O}(\mathit{A}_{\mu})$ 
is a commutative semi-group with zero element, given by the identity
\par
-ii) the mapping
\begin{equation}  %%%%%   Eqn   (15)   Begin
\mathcal{O}: \mathbb{GL}(\mathit{A}_{\mu})  \longrightarrow \mathbb{O}(\mathit{A}_{\mu} )
\end{equation}  %%%%%   Eqn   (15)   End
is additive and commutes with {\em local} conjugation
\begin{equation}
\mathcal{O}(u_{1}  + u_{2} )  =   \mathcal{O}( u_{1} )  +     \mathcal{O}( u_{2} )  
\end{equation}
\begin{equation}
\mathcal{O}(\lambda \; u \; \lambda^{-1})  =   \lambda  \; \mathcal{O}( u )   \; \lambda^{-1}, \;\; \lambda \in \mathbb{GL}(\mathit{A}_{\mu}),
\end{equation}
i.e. if $u_{1} \sim_{sl} u_{2}$ then  $\mathcal{O}(u_{1}) \sim_{sl} \mathcal{O}(u_{2})$ 
\par
-iii) 
\begin{equation}
\mathcal{O}(u^{-1}) \sim_{sl}  \mathcal{O}(u).
\end{equation}
\par
-iv)
\begin{equation}  %%%%%   Eqn   (19)   Begin
\mathcal{O}(u_{1} u_{2}) \neq   \mathcal{O}(u_{1}) \mathcal{O}(u_{2}),
\end{equation}  %%%%%   Eqn   (19)   End
\end{proposition}     %%%%%%   Proposition  17   End
%%%%%%%%%%%%%%%%%%%%%%%%%%%%%
\begin{proposition}  %%%%  Proposition  18   Begin
$\mathbb{O}(\mathit{A}_{\mu})/ \sim_{sl}$ is a sub semi-group with zero element in the commutative  semi-group  
$\mathbb{GL}(\mathit{A}_{\mu})/ \sim_{sl}$.
\end{proposition}   %%%%%   Proposition  18  End
%%%%%%%
\begin{proof}   %%%%  Proof   Proposition   18    Begin
According to Proposition 13, $\mathbb{GL}(\mathit{A}_{\mu})/ \sim_{l}$ is a commutative semi-group; the addition in this semi-group is given by the direct sum.
\par
Proposition 18 shows that $\mathbb{O}(\mathit{A}_{\mu})$ is closed under addition and that it has zero element.
\end{proof}   %%%%   Proof  Prop 18    End
%%%%%%%%%%
\begin{definition}  %%%%%%   Definition  19   Begin
Define
\begin{equation}   %%%%   Eqn  (20)   Begin
K_{1}(\mathit{A}_{\mu}) :=  \; Image:
[ (\mathbb{GL}(\mathit{A}_{\mu}) / \sim_{sl} ) / ( \;  \mathbb{O}(\mathit{A}_{\mu}) /\sim_{sl} \;)
\longrightarrow
 (\mathbb{GL}(\mathit{A}_{\mu+2}) / \sim_{sl} \; ) / ( \;  \mathbb{O}(\mathit{A}_{\mu+2}) /\sim_{sl} \;)]
.
\end{equation}   %%%%   Eqn  (20)    End
The loss of two filtration orders in the formula (20) are needed to insure that the the relation $\sim_{sl}$ behalves  transitively in the co-domain.
\end{definition}  %%%%%%   Definition  19   End
The next definition explains the meaning of the two quotients in the formula (20) by introducing
a new equivalence relation $\sim_{\mathbb{O}_{\mu}}$.
\begin{definition}   %%%%%%%   Definition  20   Begin
We say that two elements $u_{1}, u_{2} \in \mathbb{GL}(\mathit{A}_{\mu})$ are equivalent modulo the sub semi-group
 $\mathbb{O}(\mathit{A}_{\mu})/ \sim_{sl}$ (and write $u_{1}  \sim_{\mathbb{O}_{\mu}}  u_{2}$)
 provided there exist two elements $\xi_{1}, \xi_{2} \in   \mathbb{O}(\mathit{A}_{\mu})$
 such that  $u_{1} + \xi_{1} = u_{2} + \xi_{2}$.
 \end{definition}  %%%%%   Definition  20  End
 \begin{remark} %%%%%   Remark   21   Begin
 Let us denote by $J(u)$ the equivalent class of the invertible matrix $u \in \mathbb{GL}_{n}(\mathit{A}_{\mu})$ with respect to the local conjugation equivalence relation $\sim_{l}$.
 Given that the pair $(u, u^{-1})$ is a direct sum of matrices and that the Jordan canonical form $J$ acts separately onto the two components, the relation $\sim_{\mathbb{O}_{\mu}}$ kills all pairs of abstract Jordan canonical forms $(J(u), J(u^{-1}))$, for any
 $u \in \mathbb{GL}(\mathit{A}_{\mu})$.
 \par
 Note that the two components of the pair $(u, u^{-1})$ determine one each other.
 \end{remark}  %%%%%     Remark   21     End
%%%%%%%%
 \begin{proposition}  %%%%%%%   Proposition  22   Begin
 -i) 
 $\sim_{\mathbb{O}_{\mu}}$ is an equivalence relation.
 \par
 Any local inner auto-morphism induces the identity on $K_{1}^{loc} (\mathit{A})$.
 \par
 -ii) If  $ \xi \in \mathbb{O}_{2n}(\mathit{A}_{\mu})$ then $\xi \sim_{\mathbb{O}} 1_{2n}$, i.e. $\xi$ represents the zero element in $K_{1}(\mathit{A}_{\mu}) $.
 \par
 \end{proposition}  %%%%%   Proposition  22  End
 \begin{proof}   %%%%%  Proof  Prop.   22   Begin
 -i)
 $\sim_{\mathbb{O}_{\mu}}$ is reflexive. Indeed, if $u \in \mathbb{GL}_{n}(\mathit{A}_{\mu})$ and because 
 $1_{n} \in \mathbb{O}_{n}(\mathit{A}_{\mu})$, one has
  $u + 1_{n} = u + 1_{n }\in \mathbb{GL}_{2n}(\mathit{A}_{\mu})$.
 \par
$\sim_{\mathbb{O}_{\mu}}$ is symmetric: if  $u_{1}+ \xi_{1} = u_{2} + \xi_{2}$, then  $u_{2}+ \xi_{2} = u_{1} + \xi_{1}$.
\par
The transitivity is clear. If
 $$u_{1}+ \xi_{1} = u_{2} + \xi_{2} \;\;  and \;\; u_{2}+ \tilde{\xi}_{2} = u_{3} + \xi_{3}$$
  then
$$u_{1}+ (\xi_{1} + \tilde{\xi}_{2 }) = u_{2} + \xi_{2} +  \tilde{\xi}_{2 } = u_{3} +  (\xi_{2} + \xi_{3}).$$ 
\par
-ii) 
For any element $\xi \in \mathbb{O}(\mathit{A}_{\mu})$,   the element $ 1_{n} $ belongs already to  $ \mathbb{O}(\mathit{A}_{\mu})$ and
then one gets
\begin{equation} %%%%   Eqn  (21)   Begin
\xi + 1_{n} = 1_{n} + \xi,
\end{equation}  %%%%   Eqn  (21)     End
which shows that  $\xi \sim_{\mathbb{O}_{\mu}} 1_{n}$, i.e. $\xi$ represents the zero element in $K_{1}(\mathit{A}_{\mu})$.
\end{proof}   %%%%%  Proof  Prop.   22  End
% %%%%%%%%%
%%%%%%%%%%%%%%%%%%%%%%%%%%%%% 
\begin{proposition} %%%%%   Proposition 23       Begin
-i) The equivalence relation $\sim_{\mathbb{O}_{\mu}}$ is compatible with the semi-group structure: for any $u_{1}, u_{2} \in \mathbb{GL}(\mathit{A})$
 \begin{equation}  %%%%%%   Eqn  (22)   Begin
 [  u_{1} + u_{2}   ]   =   [  u_{1} ]  + [ u_{2}  ]   \in K_{1}(\mathit{A}_{\mu})
 \end{equation}    %%%%%   Eqn (22)  End
-ii) $K_{1}(\mathit{A}_{\mu})$ is a (commutative) group; for any $u \in \mathbb{GL}(\mathit{A}_{\mu})$
\begin{equation}  %%%%   Eqn  (23)   Begin
- [u] = [ u^{-1} ] \in K_{1}(\mathit{A}_{\mu}).
\end{equation}   %%%%%   Eqn  (23)  End
\par
-iii) $K_{1}(\mathit{A}_{\mu})$ and $K_{1}^{loc}(\mathit{A})$ consist of equivalence classes of invertible elements; 
the elements of this group are not {\em virtual} elements (as in the Grothendieck completion case). 
\par
-iv) Let $u_{1}, u_{2} \in \mathbb{GL}(\mathit{A}_{\mu})$ such that 
$[ u_{1}] = [u_{2}] \in K_{1}(\mathit{A}_{\mu})$. Then there exist two
elements $\xi_{1}, \xi_{2} \in \mathbb{O}(\mathit{A}_{\mu})$ and $u \in \mathbb{GL}(\mathit{A}_{\mu})$ such that
\begin{equation}   %%   Eqn  (24)   Begin
u_{1} + \xi_{1} = u\; (u_{2} + \xi_{2}) \; u^{-1} \in \mathbb{GL}(\mathit{A}_{\mu}) 
\end{equation}   %%%%   Eqn  (24)   End
\end{proposition}  %%%%%   Proposition   23     End
  %%%%%%%%%%%
%%%%%%%%%%%%%%%%%%%%%%%%%%%%%%%%%%
\begin{proof}   %%%%   Proof    Prop.  23      Begin
-i)  Suppose $u_{1} \sim_{\mathbb{O}_{\mu}} u_{1}^{'}$  and  $u_{2} \sim_{\mathbb{O}_{\mu}} u_{2}^{'}$. This means there exist four
elements  $\xi_1{},  \xi_{1}^{'}, \xi_{2}, \xi_{2}^{'} \in \mathbb{O}_{\mu}$ such that
\begin{equation}     %%   Eqn  (25)   Begin
u_{1} + \xi_{1} = u_{1}^{'} + \xi_{1}^{'},   \;\;\  u_{2} + \xi_{2} = u_{2}^{'} + \xi_{2}^{'}, 
\end{equation}    %%%%   Eqn  (25)   End
which give by addition
\begin{equation}   %%   Eqn  (26)   Begin
( u_{1} + u_{2}) + (\xi_{1} + \xi_{2})  =  ( u_{1}^{'} + u_{2}^{'}) + (\xi_{1}^{'} + \xi_{2}^{'}),  
\end{equation}    %%%%   Eqn  (26)   End
i.e.
\begin{equation}   %%%%%   Eqn  (27)   Begin
( u_{1} + u_{2})  \sim_{\mathbb{O}_{\mu}}  ( u_{1}^{'} + u_{2}^{'})) 
\end{equation}    %%%%%   Eqn  (27)
\par
--ii) Let $u \in \mathbb{GL}(\mathit{A}_{\mu})$. Then $(u + u^{-1}) + 1_{2n} = 1_{2n} + (u \oplus u^{-1}) $. In other words,
\begin{equation}  %%   Eqn  (28)   Begin
[u + u^{-1}] = [u] + [u^{-1}]  \sim_{\mathbb{O}_{\mu}}  1_{2n} \sim_{\mathbb{O}_{\mu}}  0 \in  K_{1}(\mathit{A}_{\mu}).
\end{equation}    %%%%   Eqn  (28)   End
\par
-iii) It is clear.
\par
-iv) Once the factorisation through the semi-group $\mathbb{O}(\mathit{A}_{\mu})$ is used to define
$K_{1}(\mathit{A}_{\mu})$, it remains to involve the equivalence relation $\sim_{l}$ (note  $\sim_{l}$ and $\sim_{s}$ commute).
\end{proof}   %%%%%   Proof   Prop.  23  End 

%%%%%%
  \begin{remark}   %%%%%%   Remark  24   Begin
  The factorisation used in the formula (20) for the definition of $K_{1}$ is weaker than the factorisation used in
   the classical definition of the  $K_{1}$-theory group. In fact, the identity (23) is equivalent to the factorisation (20); on the other side, formula (23) holds in the classical $K$-theory, see e.g.  \cite{Milnor}.
\end{remark}  %%%%%   Remark   24   End
  %%%%%%%%
\par
We may be more explicit on this point. In the classical definition 
\begin{equation}   %%%%%   Eqn   (29)   Begin
K_{1} (\mathit{A}) := \mathbb{GL} (\mathit{A}) / [ \mathbb{GL} (\mathit{A}), \mathbb{GL} (\mathit{A}) ].
\end{equation}  %%%%%   Eqn   (29)  End
On the other side, the following identity holds
%%%%%%%%%%%%%%%%%%%%%%%%%%%%
\begin{equation}   %%%%   Begin  Eqn  (30)
\begin{pmatrix}
A  &  0\\
0  &  B
\end{pmatrix}
%%%%%%
=
%%%%%%
\begin{pmatrix}
AB & 0\\
0  &  1
\end{pmatrix}  
%%%%%%%    
\begin{pmatrix}
B^{-1} & 0\\
0  &  B
\end{pmatrix}    
, 
\end{equation} 
%%%%   Begin  Eqn  (30) 
which shows that multiplicatively, modulo elements of the form $\mathcal{O}(B)$ (which belong to the commutator sub-group
$\mathbb{O}(\mathcal{A})$), one has
\begin{equation}  %%%%   Eqn (31)   Begin
[A + B] = [AB] \in K_{1}(\mathit{A}).
\end{equation}  %%%   Eqn   (31)   End
\par
\begin{proposition}  %%%%%   Proposition  (25)   Begin
-i) The relation (23) holds in  $K_{1}(\mathit{A}_{\mu})$ and $K_{1}^{loc}(\mathit{A})$. 
\par
-ii) The {\em local} algebraic $K_{i}(\mathit{A}_{\mu})$ and   $K_{i}(\mathit{A}^{loc})$ groups are Morita invariant.
\end{proposition}    %%%%%%    Proposition    ( 25)   End  
\begin{proof}  %%%%%   Proof  Proposition  (25)   Begin
The projective limit with respect to $\mu$ assures the formula (23) remains valid in $K_{1}^{loc}(\mathit{A})$.  
\end{proof}    %%%%%   Proof  Proposition  (25)     End
 %%%%%%%%%%%%%%%%%%%%%%%%%%%%%%%%%%%%  
\section{Induced homomorphisms.}   %%%%%%   Sect   7   Induced homomorphisms
 \begin{definition}   %%%%%  Definition  26  Begin
Let $f:  \mathit{A} \longrightarrow  \mathit{B}$  be a localised algebra homomorphism.
\par
 Then the homomorphism $f_{\ast}$ (see Definition 5) induces homomorphisms
\begin{equation}   %%%%   Eqn  (32)   Begin
f_{\ast}:  K_{0} (\mathit{A}_{\mu})   \longrightarrow  K_{0} (\mathit{B}_{\mu}),  \hspace{0.5cm}
f_{\ast}:  K_{0}^{loc} (\mathit{A})  \longrightarrow  K_{0}^{loc} (\mathit{B})
\end{equation}  %%%%   Eqn  (32)   End
and
\begin{equation}   %%%%   Eqn  (33)   Begin
f_{\ast}:  K_{0} (\mathit{A}_{\mu})   \longrightarrow  K_{0} (\mathit{B}_{\mu}),  \hspace{0.5cm}
f_{\ast}:  K_{1}^{loc} (\mathit{A})  \longrightarrow  K_{1}^{loc} (\mathit{B})
\end{equation}  %%%%   Eqn  (33)   End
 \end{definition}   %%%%%%   Definition  26  End
%%%%%%%%%%%%%%%%%%%%%%%%%%%%%%%% 
\section{Constructing idempotents and invertible matrices over $\Lambda_{\mu}$.}  %%%%   Sect  8. Constructing K_{i}(\Lambda)
%%%%%%%%%%%%%%%%%%%%%%%%%%%%%%%% 
We come back to the situation presented in \S  5 - Mayer-Vietoris diagrams. 
\subsection{Constructing idempotents over $\Lambda_{\mu}$.} 
\begin{theorem} {\bf Definition}.  %%%%%%%  Begin  Definition-Theorem 27
Suppose $p_{1} \in \mathbb{I}demp_{n}(\Lambda_{1, \mu})$, resp.  $p_{2} \in \mathbb{I}demp_{n}(\Lambda_{2, \mu})$,  is an idempotent matrix with entries in
$\Lambda_{1, \mu}$, resp. $\Lambda_{2, \mu}$, such that
\begin{equation}  %%%   Begin  Eqn  (34)
j_{1\ast} (p_{1}) =  u \;  j_{2\ast} (p_{2}) \; u^{-1},
\end{equation}  %%%   End  Eqn  (34)
where $u \in \mathbb{GL}_{n}(\Lambda'_{\mu})$ is an invertible matrix.
\par
-i) Then there exists an idempotent double matrix $p \in \mathbb{I}demp_{2n}(\Lambda_{\mu})$ such that
\begin{equation}  %%%  Begin  Eqn  (35)
i_{1\ast} (p) = p_{1} \oplus 0_{n}  \;\;  and  \;\; \; i_{2\ast} (p) = \tilde{p}_{2} 
\end{equation}  %%%%  End  Eqn  (35)
where the idempotent $\tilde{p}_{2} \in M_{2n}(\Lambda_{2, \mu})$ is conjugated to $p_{2} \oplus 0_{n}$ through an invertible matrix $\tilde{U} \in GL_{2n} (\Lambda_{2, \mu})$, that is  $\tilde{p}_{2} = \tilde{U} p_{2}  \tilde{U}^{-1}$.
\par
The corresponding double matrix idempotent $p$ is denoted $p =(p_{1}, p_{2}, u)$. 
\par
-ii) 
\begin{equation}   %%%    Eqn  (36)   Begin
j_{2, \ast} \tilde{p}_{2} = (u \; j_{2, \ast} (p_{2} ) \; u^{-1}) \oplus 0_{n} = j_{1, \ast} ({p}_{1} \oplus 0_{n})
\end{equation}   %%%    Eqn  (36)   End
\par
-iii) Any idempotent matrix $p \in \mathbb{M}_{2n} (\Lambda_{\mu})$  is conjugated through a localised inner automorphism 
defined by a matrix $U \in \mathbb{GL}_{2n}(\Lambda_{\mu})$  to
an idempotent matrix of the form $p =(p_{1}, p_{2}, u)$ defined by -i).
\end{theorem} %%%%%%%%   End Definition-Theorem  27
%%%%%%%%%%%%%%%%%%%%%%%%%%%%%%%%%%%%%%%%%%%%%%%%%%%%%%
Condition (34) says that $[j_{1 \ast }p_{1} ] = [j_{ \ast }p_{2} ] \in K_{0} (\Lambda'_{\mu})$.
Part -i) says that the pair  $([p_{1}], [\tilde{p}_{2}]) \in K_{0}(\Lambda_{1})  \oplus K_{0}(\Lambda_{2}) $ belongs to the image of
$(i_{1 \ast},  i_{2 \ast})$. This property is part of the proof of Theorem 51.-i) on exactness.
\par
%%%%%%%%%%%%%%%%%%%%%%%%%%%%%%%%%%%%%%%%%%%%
\begin{proof}  %%%%%%  Begin  Proof   Theorem  27
All considerations made to prove this theorem will be performed onto objects related to $\Lambda_{2, \mu}$.
%%%%%%%%%%%%%
\begin{lemma}  %%%%%%  Begin  Lemma 28
Let $p_{1} = ( a_{ij} ) \in \mathbb{I}demp_{n} (\Lambda_{1, \mu})$  and $p_{2} = ( b_{ij} ) \in \mathbb{I}demp_{n} (\Lambda_{2, \mu})$ be idempotents.
\par
Suppose the idempotents $ j_{1\ast} (p_{1})$,  $ j_{2\ast} (p_{2}) $ are conjugate through an inner automorphism defined by
$u \in \mathbb{GL}_{n} (\Lambda'_{\mu})$, i.e.
\begin{equation}   %%%%% Begin   Eqn  (37)
j_{1\ast} (p_{1}) =  u \;  j_{2\ast} (p_{2}) \; u^{-1}.
\end{equation}  %%%  End   Eqn  (37)
Assume, additionally, that the invertible element $u$ lifts to an invertible element 
$\tilde{u} \in \mathbb{GL}_{n}(\Lambda_{2, \mu})$ 
(i.e. $j_{2\ast} \tilde{u} = u$).
\par
Then $p =(p_{1}, p^{'}_{2}, u)   \in  \mathbb{I}demp_{n} (\Lambda \mu)  $ is an idempotent given by the double matrix
\begin{equation}  %%%%%  Begin  Eqn  (38)
p = ( (a_{ij}, c_{ij})) ,
\end{equation}   %%%%  End  Eqn  (38)
where
\begin{equation}   %%%%%%  Begin  Eqn  (39)
 (a_{ij}) = p_{1} \in \mathbb{I}demp_{n}(\Lambda_{1, \mu}) \;\; and \;\; p^{'}_{2} = ( c_{ij} )  := 
 \tilde{u} \;   p_{2} \; \tilde{u}^{-1} \in \mathbb{I}demp_{n}  ( \Lambda_{2,  \mu}).
\end{equation}   %%%   End  Eqn  (39)  
\end{lemma}   %%%%%%    End  Lemma  28 
Remark that in this lemma the size of the double matrix $p$ does not change. 
\par
%%%%%%%%%%%%%%%%%%%%%%%%%%%%%%%%%%%%%%%%%%%%%%
%% \begin{proof}          %%%%%%  
$\mathit{ Proof \; of \; Lemma} \; 28$.  %%%%%  Begin Proof Lemma 28
We use Remark 8.-ii).
It is clear that the matrix $p$ given by (38), (39) is an idempotent. In fact, to evaluate $p^{2}$ amounts to compute separately  the square of the first and second component matrices of the matrix $p$,  i.e. the squares of $(a_{ij})$ and $(c_{ij})$. These are 
\begin{equation}   %%%%%%%  Begin  Eqn (40)
(a_{ij})^{2} = (a_{ij})  \;\;and\;\;
(c_{ij})^{2} =  (\; \tilde{u}  \; p_{2} \; \tilde{u}^{-1} \;)^{2} =  \tilde{u}  \; p_{2}^{2} \; \tilde{u}^{-1} = \tilde{u}  \; p_{2} \; \tilde{u}^{-1} = (c_{ij}).
\end{equation}   %%%%%%  End  Eqn  (40)
\par
It remains to verify that  $p \in \mathbb{M}_{n} (\Lambda_{\mu})$, i.e.
$j_{1 \ast} (a_{ij}) = j_{2 \ast} (c_{ij}) $.  This follows from (40) combined with (37)
\begin{equation}  %%%%  Begin  Eqn  (41)
j_{1\ast} (p_{1}) =  u \;  j_{2\ast} (p_{2}) \; u^{-1} = j_{2 \ast} (\tilde{u} \;  p_{2} \; \tilde{u}^{-1} ) =  j_{2 \ast} (\tilde{p}_{2}).  
\end{equation}  %%%%  End  Eqn  (41)
This ends the proof of Lemma 28.
%%%%%%%%%%%%%%%%%%%%%%%%%%   End  Proof  Lemma  28
%%%%%%%%%%%%%%%%%%%%%%%%%%%%%%%%%%%%%%%%%%
\begin{lemma}  %%%%%   Begin  Lemma 29
Let $p_{1} = ( a_{ij} ) \in \mathbb{I}demp_{n} (\Lambda_{1, \mu})$  and $p_{2} = ( b_{ij} ) \in \mathbb{I}demp_{n} (\Lambda_{2, \mu})$ be idempotents.
\par
Suppose the idempotents $ j_{1\ast} (p_{1})$,  $ j_{2\ast} (p_{2}) $ are conjugate through an inner automorphism defined by
$u \in \mathbb{GL}_{n} (\Lambda^{'} _{\mu})$, i.e.
\begin{equation}   %%%%%   Eqn  (42)    Begin
j_{1\ast} (p_{1}) =  u \;  j_{2\ast} (p_{2}) \; u^{-1}.
\end{equation}  %%%    Eqn  (42)    End
Then
\par
-i) $ j_{1 \ast}(p_{1} \oplus 0_{n}) $ and $ j_{2 \ast}(p_{2} \oplus 0_{n}) $ are conjugated by 
$U := u \oplus u^{-1}  \in  \mathbb{GL}_{2n}(\Lambda'_{\mu}) $,  i.e.
\begin{equation}  %%%%%   Begin  Eqn  (43)
 j_{1 \ast}(p_{1}) \oplus 0_{n} =
 j_{1 \ast}(p_{1} \oplus 0_{n})  = U \;  j_{2 \ast}(p_{2} \oplus 0_{n}) \; U^{-1} =  (\; u  \; j_{2 \ast}(p_{2}) \;  u^{-1} \; ) \oplus 0_{n}.
\end{equation}  %%%%%  End  Eqn  (43)
\par
-ii) Supposing that $j_{2}$ is surjective, the invertible matrix $U$ lifts to an invertible matrix $\tilde{U} \in \mathbb{M}_{2n, \mu} (\Lambda_{2})$. Let 
\begin{equation}  %%%%%   Eqn  (44)  Begin
\tilde{p}_{2} := \tilde{U}   \; p_{2}  \;  \tilde{U}^{-1}.
\end{equation}    %%%%%   Eqn  (44)   End
Then
\begin{gather} %%%%%   Eqn  (45)  Begin
j_{1, \ast} (p_{1} \oplus 0_{n}) = (u \; j_{2, \ast} ({p}_{2}) \; u^{-1})\oplus 0_{n} = j_{2, \ast} (\tilde{p}_{2} );
\end{gather}  %%%%%   Eqn  (45)   End
i.e.  the matrices $p_{1} \oplus 0_{n}, \; \tilde{p}_{2}$ form a double matrix idempotent in 
$\mathbb{I}demp_{2n}(\Lambda_{\mu})$, denoted $p := (p_{1}, p_{2}, u)  \in  \mathbb{I}demp_{2n}(\Lambda_{\mu})  $ and
 \begin{equation}  %%%%%   Eqn  (46)  Begin
 (i_{1 \ast}, i_{2 \ast} ) p = (p_{1}, \tilde{p}_{2}) 
 \end{equation}  %%%%%   Eqn  (46)  Begin
\par
-iii) The pair of idempotents  $p_{1}  \in \mathbb{I}demp_{n} (\Lambda_{1, \mu})$,   $p_{2}  \in \mathbb{I}demp_{n} (\Lambda_{2, \mu})$  is stably equivalent to the pair of idempotents $p_{1} \oplus 0_{n}$,  ${p}_{2} \oplus 0_{n}$ and ${p}_{2} \oplus 0_{n} \sim_{l}  \tilde{p}_{2}$. In other words
 \begin{equation}  %%%%%   Eqn  (47)  Begin 
     ([p_{1}], [{p}_{2}]) =  ([p_{1}], [\tilde{p}_{2}])   \in K_{0} (\Lambda_{1, \mu }) \oplus K_{0}(\Lambda_{2}, \mu) .
 \end{equation}   %%%%%   Eqn  (47)   End
 %%%%%%%%%%%%%%%%
 \end{lemma}   %%%%%%  End  Lemma  29
 Note that in this lemma, in comparison with the preceding Lemma  28, the size of the desired idempotent doubles; otherwise,
 the important modifications still occur onto matrices associated with $\Lambda_{2, \mu}$. 
 \par
 $\mathit{Proof \; of \; Lemma \; 29}$.   %%%%%   Begin  Proof Lemma 29 
 Part -i) is clear.
 \par
 The proof of -ii) uses $\mathcal{O}_{2n}(u)$, where $u \in \mathbb{GL}_{n}(\mathit{A}_{\mu})$, 
 Definition 15. The next formula  provides 
 a decomposition of  $\mathcal{O}_{2n}(u)$ in a product of elementary matrices and a scalar matrix
\begin{equation}  %%%%%%%   Begin  Eqn  (48)
U := \mathcal{O}_{2n}(u) = 
\begin{pmatrix}
u & 0 \\
0 & u^{-1}
\end{pmatrix}
=
%%%%%
\begin{pmatrix}
1 & u \\
0 & 1
\end{pmatrix}
%%%%%%
\begin{pmatrix}
1 & 0 \\
- u^{-1} & 1
\end{pmatrix}
%%%%%%%
\begin{pmatrix}
1 & u \\
0 & 1
\end{pmatrix}
%%%%%%%%
\begin{pmatrix}
0 & -1 \\
1 & 0
\end{pmatrix}
\in \mathbb{M}_{2n}(\Lambda_{2, \mu}).
\end{equation}    %%%%%%%   End  Eqn  (48)
\par
To complete the proof we need to show that the invertible matrix $U$ has an invertible lifting $\tilde{U} \in \Lambda_{2, \mu}$. This follows from the properties of elementary matrices, discussed next.
\begin{definition}  Elementary Matrices. %%%%   Begin   Definition  30    Elementary Matrices
\par
A matrix $E_{ij} (a) \in \mathit{GL}_{n}(\mathit{A}_{\mu}) $ having all entries equal to zero, except for the diagonal entries equal to $1$ and just one $(i,j)$-entry $a \in \mathit{A}_{\mu}$,  $0 \leq i \neq j \leq n$  is called  {\em elementary matrix with entry in} $\mathit{A}_{\mu}$. 
\par
The space of {\em elementary matrices with entries in} $\mathit{A}_{\mu}$ is by definition 
\begin{equation}  %%%%%  Eqn  (49)   Begin
\mathit{E}_{n} (\mathit{A}_{\mu}) \; := \{  E_{ij} (a) \; | \; 1 \leq i \neq j \leq n, \; a \in \mathit{A}_{\mu} \}.
\end{equation}   %%%%%  Eqn  (49)   End
Let $\mathbb{E}_{n} (\mathit{A})$ the sub-group generated by all elementary matrices and
\begin{equation}   %%%%%  Eqn  (50)   Begin
\mathbb{E}(\mathit{A}) := \injlim_{n \in \mathbb{N}} \mathbb{E}_{n} (\mathit{A}).
\end{equation}    %%%%%  Eqn  (50)   End
\end{definition}  %%%%   End   Definition  30   End
\par
%%%%%%%%%%%%%%%%%%%%%%%%%%%%%%%%%%%%%%%%
\begin{lemma}  %%%%%%    Lemma  31  Begin
The elementary matrices satisfy
\begin{equation}   %%%%%  Eqn  (51)   Begin
E_{i,j}(a) . E_{i,j}(b)  \; = \; E_{i,j}(a+b)  
\end{equation}  %%%%%  Eqn  (51)   End
\begin{equation}  %%%%%  Eqn  (52)   Begin
E_{i,j}(a)^{-1}  \; = \; E_{i,j}(-a),  
\end{equation}  %%%%%  Eqn  (52)   End
therefore $E_{i,j}(a)  \in \mathit{GL} (\mathit{A}_{\mu})$,
 Any elementary matrix is a commutator
\begin{equation}  %%%%%%%   Begin  Eqn  (53)
E_{ij} (A) = [ E_{ik}(A), E_{kj}(1)],   \hspace{0.2cm}  for \; any\; i, j, k \; distinct \; indices.
\end{equation}  %%%%%%  End  Eqn  (53)
\end{lemma}   %%%%%%    Lemma   31  End
%%%%%%%%%%%%%%%%%%%%%%%%%
\par
We come back to the proof of Lemma 29. -ii). 
%%%%%%%%%%%%%%%%%%%%%%%%%%%%%%%%%%%%%%%
As $j_{2}$ is surjective, each of the entries of factors of the RHS of (48) has a lifting in $\mathbb{M}_{2n}(\Lambda_{2 \mu})$ so as each of the elementary matrix factor lifts as an invertible elementary matrix.  The last factor lifts as it is. 
Therefore, $U$ has an invertible lifting $\tilde{U} \in \mathbb{GL}_{2n}(\Lambda_{2, \mu})$.   
Lemma 28 completes the proof of Lemma 29.  -ii).  
\par
Lemma 29. -iii) follows from the definition of $K^{loc}_{0}(\Lambda)$.
This completes the proof of Lemma 29.
%%%%%%   End  Proof  Lemma 29
%%%%%%%%%%%%%%%%%%%%%%%%%%%%%%    
\par
Theorem 27 follows from Lemma 28 combined with Lemma 29.
\end{proof}   %%%%%   End  Proof Theorem 27.
%%%%%%%%%%%%%%%%%%%%%%%%%%%%%%% 
\par
Theorem 27 refers to the construction and description of idempotents over $\Lambda_{\mu}$. 
We need to extend Theorem 27 to elements of $K_{0}^{loc}(\mathit{A}),$  i.e. to formal differences of {\em local} idempotents. 
A new difficulty occurs. This has to do with the ambiguity involved in the Grothendieck construction; without special precautions, we would have to lift idempotents over $\Lambda^{'}_{\mu}$ to idempotents over  
 $\Lambda_{2, \mu}$ which, in general, could not be done.
 \par
Now we proceed.
\begin{lemma}    (compare  \cite{Milnor}  Lemma 1.1)  %%%%%%   Lemma 32   Begin
\par
Let $p_{1}, p_{2} \; q_{1}, q_{2} \in \mathbb{I}demp_{n}(\mathit{A}_{\mu})$ be idempotents and let $[\;\;]$ denote their $K_{0}(\mathit{A}_{\mu})$ class.
Suppose
\begin{equation}
[p_{1}] - [p_{2}] \;=\; [q_{1}] - [q_{2}] \in K_{0}(\mathit{A}_{\mu}).
\end{equation}
\par
Then $p_{1} + q_{2}$ and $p_{2} + q_{1}$ are {\em local stably} isomorphic,  $p_{1} + q_{2} \sim_{ls} p_{2} + q_{1}$.
\end{lemma}   %%%%%    Lemma 32   End
\begin{proof} %%%%%%    Proof   Lemma  32  Begin
As said above, the description of  $K_{0}(\mathit{A}_{\mu})$ classes in terms of generators (idempotents) contains an ambiguity due to the Grothendieck completion construction and stabilisation. The Grothendieck completion implies that there exists an idempotent $s \in  \mathbb{I}demp_{m}(\mathit{A}_{\mu})$ such that the idempotents 
$$p_{1} + q_{2} + s,   \;\;\;\;  p_{2} + q_{1}+ s$$ 
are {\em local} stably isomorphic. We assume that the idempotent $s$ is already sufficiently stabilised.
This means there exists an invertible matrix $u \in \mathbb{GL}_{2n+m}(\mathit{A}_{\mu})$ 
such that
$$p_{1} + q_{2} + s,   \; = \; u \; ( p_{2} + q_{1}+ s ) \; u^{-1}.$$ 
We add to both sides of this equality the idempotent $1_{m} - s$ and we extend $u$ to be the identity on the last summand. 
We get  
$$p_{1} + q_{2} + s + (1-s) ,   \; = \; u \; ( p_{2} + q_{1}+ s +(1-s) ) \; u^{-1}.$$ 

From this we get further
$$p_{1} + q_{2} + 1_{2m} ,   \; = \; u \; ( p_{2} + q_{1} + 1_{2m} ) \; u^{-1},$$ 
that is, the idempotents $p_{1} + q_{2} $,     $p_{2} + q_{1} $ are {\em local stably} isomorphic
$$p_{1} + q_{2} \; \sim_{sl}    p_{2} + q_{1}. $$ 
\end{proof}   %%%%%%   Proof Lemma  32   End
%%%%%%%%%%%%%%%%%%%%%%%%%%%%%%%%%%%
\begin{lemma}  %%%%%%   Lemma 33   Begin
Let $p, q \in \mathbb{I}demp_{n}(\mathit{A}_{\mu})$ be idempotents.
Suppose
\begin{equation}
[p] - [1_{n}] \;=\; [q] - [1_{n}] \in K_{0}(\mathit{A}_{\mu}).
\end{equation}
Then
\par
-i)  $p$ and $q$ are {\em local stably} isomorphic,  $p \sim_{ls} q$,
\par
-ii)  exists an $N \in \mathbb{N}$ and an $u \in \mathbb{GL}_{n+N}(\mathit{A}_{\mu})$ such that
\begin{equation}  %%%%%%   Eqn (56)   Begin
p + 1_{N} = u \; q \; u^{-1} + 1_{N} =  u \; (q + 1_{N})\; u^{-1}.
\end{equation}  %%%%%   Eqn  (56)  End
\end{lemma}   %%%%%    Lemma 33   End
\begin{proof}   %%%%   Proof Lemma  33  Begin
-i) Lemma 32 says that  the idempotents $p + 1_{n}$,  $q + 1_{n}$ are {\em local stably} isomorphic. This means that the idempotents
$p$ and  $q$ are {\em local stably} isomorphic. Part -ii) describes this.
\end{proof}  %%%%%   Proof  Lemma  33  End

%%%%%%%%%%%%%%%%%%%%%%%%%%%%%%%%%%%
\begin{theorem}   %%%%%%%   Theorem  34  Begin
Let $p_{ij}$ be idempotents 
$$[p_{1}] = [p_{11}] - [p_{12}]       \in K_{0}(\Lambda_{1, \mu})$$  
$$[p_{2}] = [p_{21}] - [p_{22}]      \in K_{0}(\Lambda_{2, \mu})$$
\par
with the property that
$$j_{1 \ast} [p_{1}] = j_{2 \ast} [p_{2}] \in K_{0}(\Lambda^{'}_{\mu}).$$
\par
Then there exists $[p] = [p_{01}] - [p_{02}] \in K_{0}(\Lambda_{\mu})$ with the property that
$$i_{1 \ast} [p] = [p_{1}] \;\; and \;\; i_{2 \ast} [p] = [p_{2}].$$
\end{theorem}  %%%%%%%   Theorem  34   End
%%%%%%%%%%%%%%%%%%%%%%%%%%%%
\begin{proof}   %%%%%%%   Proof  Theorem  34  Begin   
We may describe  the two $K$-theory classes differently

$$[p_{1}] = [p_{11}] - [p_{12}] =     [p_{11}+ (1 - p_{12} )] - [p_{12} + (1-p_{12})] = [p^{'}_{12}] - [ 1_{n} ]       \in K_{0}(\Lambda_{1, \mu})$$  
and
$$[p_{2}] = [p_{21}] - [p_{22}]  = [p_{21}+ (1 - p_{22}) ] - [p_{22} + (1-p_{22})] = [p^{'}_{22}] - [ 1_{n} ]       \in K_{0}(\Lambda_{2, \mu}).$$
\par
Then
$$j_{1 \ast} [p_{1}] =  j_{1 \ast }(  [p^{'}_{12}] - [ 1_{n} ] ) = ( j_{1 \ast }  [p^{'}_{12}]) - [ 1_{n} ] $$ 
and 
$$ j_{2 \ast} [p_{2}]  =  j_{2 \ast }(  [p^{'}_{22}] - [ 1_{n} ] ) = ( j_{2 \ast }  [p^{'}_{22}]) - [ 1_{n} ]. $$ 
The hypothesis says that 
$$ ( j_{1 \ast }  [p^{'}_{12}]) - [ 1_{n} ] = ( j_{2 \ast }  [p^{'}_{22}]) - [ 1_{n} ]. $$ 
\par
Lemma 33 says that the idempotents $ j_{1 \ast }  [p^{'}_{12}]$    $ j_{2 \ast }  [p^{'}_{22}]$  are {\em local stably} isomorphic. Now we are in the position to use Theorem 27. Let $u \in \mathbb{GL}(\Lambda_{2}(\mathit{A}_{\mu}))$ be the conjugation
\begin{equation}  %%%%%%   Eqn  (57)  Begin
 j_{1 \ast }  (p^{'}_{12}) = u  \; j_{2 \ast }  (p^{'}_{22}) \; u^{-1}.
\end{equation}  %%%%%%   Eqn  (57)   End
Theorem 27 provides the idempotent $$p = (j_{1 \ast}  (p^{'}_{12}), (j_{2 \ast}  (p^{'}_{2}),u). $$
The desired idempotents are
$$
p_{10} =  p = (j_{1 \ast}  (p^{'}_{12}), (j_{2 \ast}  (p^{'}_{2}),u) \in \mathbb{I}demp(\Lambda_{\mu})
$$
$$
p_{20} = 1_{N}   \in \mathbb{I}demp(\Lambda_{\mu}).
$$
\end{proof}   %%%%%%   End  Proof  Theorem  34
%%%%%%%%%%%%%%%%%%%%%%%%%%%%%%%%%%%%%%%%%%

\subsection{Constructing invertible matrices over $\Lambda_{\mu}$.}  %%%%   Sect  8.2. Constructing K_{1}(\Lambda)
\par
\begin{theorem} {\bf Definition}.  %%%%%%%  Begin  Definition-Theorem 35   FOR  K_{1}    BEGIN
Suppose $s_{1} \in \mathbb{GL}_{n}(\Lambda_{1, \mu})$, resp.  $s_{2} \in \mathbb{GL}_{n}(\Lambda_{2, \mu})$,  are 
invertible matrices with entries in $\Lambda_{1, \mu}$, resp. $\Lambda_{2, \mu}$, such that
\begin{equation}  %%%   Begin  Eqn  (58)
j_{1\ast} (u_{1}) =  u \;  j_{2\ast} (u_{2}) \; u^{-1},
\end{equation}  %%%   End  Eqn  (58)
where $u \in \mathbb{GL}_{n}(\Lambda'_{\mu})$ is an invertible matrix.
\par
-i) Then there exists an invertible matrix $s \in \mathbb{GL}_{2n}(\Lambda_{\mu})$ such that
\begin{equation}  %%%  Begin  Eqn  (59)
i_{1\ast} (s) = s_{1} \oplus 1_{n}  \;\;  and  \;\; \; i_{2\ast} (s) = \tilde{s}_{2} 
\end{equation}  %%%%  End  Eqn  (59)
where the invertible matrix $\tilde{s}_{2} \in \mathbb{GL}_{2n}(\Lambda_{2, \mu})$ is conjugated to $s_{2} \oplus 1_{n}$ through an inner auto-morphism defined by the invertible matrix $\tilde{U} \in GL_{2n} (\Lambda_{2, \mu})$, that is  $\tilde{p}_{2} = \tilde{U} p_{2}  \tilde{U}^{-1}$.
\par
The corresponding invertible double matrix  $s$ is denoted $s =(s_{1}, s_{2}, u)$. 
\par
-ii) 
\begin{equation}   %%%  Begin  Eqn  (60)
j_{2, \ast} \tilde{s}_{2} = (u \; j_{2, \ast} (s_{2} ) \; u^{-1}) \oplus 1_{n} = j_{1, \ast} ({s}_{1} \oplus 1_{n})
\end{equation}    %%%  End  Eqn  (60)
\par
-iii) Any invertible double matrix $s \in \mathbb{GL}_{2n} (\Lambda_{\mu})$  is conjugated through a localised inner automorphism defined by a matrix $U \in \mathbb{GL}_{2n}(\Lambda_{\mu})$  to
an idempotent matrix of the form $s =(s_{1}, s_{2}, u)$ defined by -i).
\end{theorem} %%%%%%%%   Definition-Theorem  22  For K^{loc}_{1}     END
%%%%%%%%%%%
%\end{document}
\begin{proof}    %%%%%%   Proof   Theorem  35   Begin
The proof of this theorem goes along the same manner as the proof of Theorem 27. 
\par
In retrospect, the proof of Theorem 27
is based on the following facts: -a) operations with double matrices respect Remark 8, -b) lifting of the invertible element 
$U= \mathbb{O}_{2n}(u) := u \oplus u^{-1} \in \mathbb{GL}_{2n}(\Lambda_{2, \mu})$ associated to the invertible element
$u \in \mathbb{GL}_{n}(\Lambda_{2, \mu})$ by means of the factorisation of $U$ by elementary matrices (identity (48)),
-c) the fact that the inner auto-morphisms keep the mapping $\mathbb{O}_{n}$  unchanged. 
\par
To prove Theorem 35 we use the same arguments -a), - b), -c) with the following changes: idempotents are replaced by invertible elements while for -c) we use the fact that the inner auto-morphisms transform the mapping $1_{n}$ into itself. 
\end{proof}
\par
The next theorem is the analogue of Theorem 34 in the $K_{1}(\mathit{A}_{\mu})$ case.
%%%%%%%%%%%%%%%%%%%%%%%%%%%%%%%%%%%%%%%%%%%%%%%
\begin{theorem}  %%%%%%   Theorem  36  Begin
Suppose $j_{1}$ and $j_{2}$ are epi-morphisms.
\par
Let $u_{1} \in \mathbb{GL}(\Lambda_{1, \mu})$ and  $u_{2} \in \mathbb{GL}(\Lambda_{2, \mu})$ be such that
\begin{equation}   %%%%%%%   Eqn   (61)   Begin
j_{1 \ast}  [u_{1}]  =   j_{2 \ast}  [u_{2}] \in K_{1} (\Lambda^{'}_{\mu})  \otimes \mathbb{Z}[\frac{1}{2}].
\end{equation}  %%%%%%   Eqn   (61)   End
Then there exists $s \in \mathbb{I}demp(\Lambda_{\mu})$ such that
\begin{equation}   %%%%%%%   Eqn   (62)   Begin
 i_{1 \ast} [ s ] = [u_{1}] \in K_{1}(\Lambda_{1, \mu})  \otimes \mathbb{Z}[\frac{1}{2}] \;\;\;  and \;\;\;  i_{2 \ast} [ s ] = [u_{2}]  \in K_{1}(\Lambda_{2, \mu})  \otimes \mathbb{Z}[\frac{1}{2}]. 
\end{equation}  %%%%%%   Eqn   (62)   End
\end{theorem}   %%%%%%   Theorem  36    End
%%%%%%%%%%%%%%%%%%%%%%%%%%%%%%%%%%%%%%%%%%%%%%%%
\begin{remark}  %%%%%   Remark  37  Begin
Tensor multiplication  by $\mathbb{Z}[\frac{1}{2}]$ reflects the difficulty to lift elements belonging to
$\mathcal{O}(\Lambda^{'}_{\mu})$ to elements  belonging to  $\mathcal{O}(\Lambda^{'}_{1, \mu})$ and 
$\mathcal{O}(\Lambda^{'}_{2, \mu})$.  The nature of elements belonging to $\mathcal{O}(\Lambda^{'}_{\mu})$, see formula (48), insures the existence of the lift as an invertible element; the factor $1/2$ is needed to insure that the invertible lifts belong to
$\mathcal{O}(\Lambda^{'}_{1, \mu})$ and $\mathcal{O}(\Lambda^{'}_{2, \mu})$.
\par
The presence of the factor $1/2$ has important  consequences.
\end{remark}  %%%%%   Remark  37   End
%%%%%%%%%%%%
     
\begin{proof}   %%%%%%    Proof Theorem  36   Begin
The definition of $K_{1}(\Lambda^{'}_{\mu})$, see Definition 19 - 20,  tells that the description of its elements contains an ambiguity belonging to the sub-module $\mathbb{O}(\Lambda^{'}_{\mu})$. 
 Equality (61) tells then that there exist two elements 
$\xi_{1}, \xi_{2} \in \mathbb{O}_{2n}(\Lambda^{'}_{\mu})$ such that the invertible matrices
$j_{1 \ast}(u_{1}) + \xi_{1}$, $j_{2 \ast}(u_{2}) + \xi_{2}$ are {\em local} conjugated by means of a
matrix   $u \in \mathbb{GL}_{2n}(\Lambda^{'})$
\begin{equation}  %%%%%%%   Eqn   (63)   Begin
j_{1 \ast}(u_{1}) + \xi_{1} = u \; (j_{2 \ast}(u_{2}) + \xi_{2}) \;  u^{-1}.
\end{equation}    %%%%%%   Eqn   (63)   End
At this point a new difficulty occurs. We would like to use the analogue of Lemma 29 and use the invertible matrix 
$U = u \oplus u^{-1} \in \mathbb{O}_{2n}(\Lambda^{'}_{\mu})$ to lift and twist invertible elements belonging to 
$\mathbb{GL}(\Lambda^{'}_{\mu})$. 
With the exception of $u$, in order to apply Lemma 29, we have to start with invertible elements which
are already defined over $\Lambda_{1, \mu}$  and $\Lambda_{2, \mu}$. Unfortunately,  $\xi_{1}$ and $\xi_{2}$ do not belong to these spaces. To proceed, however, along this way
we need to show: -1) that the elements $\xi_{1}, \xi_{2}$ have invertible lifts in
$\Lambda_{1, \mu}$  resp. $\Lambda_{2, \mu}$ and that -2) such lifts belong to
$\mathbb{O}_{2n}(\Lambda_{i, \mu})$,  $i = 1, 2.$  
\par
Here we address these two problems. We illustrate the solution in the case of the element $\xi_{1}$.
\par
We intend to lift the invertible matrix 
\begin{equation}  %%%%%%   Eqn  (64)   Begin 
\xi_{1} =
\begin{pmatrix}
\alpha_{1}  &  0 \\
0  &   \alpha_{1}^{-1}
\end{pmatrix}
\in \mathbb{O}_{2n}(\Lambda^{'}_{\mu})
\end{equation}  %%%%%    Eqn  (64)
to an invertible matrix $\tilde{\xi}_{1} \in  \mathbb{O}_{2n}(\Lambda_{1,  \mu})$.
%%%%%%%%%%%%%%%%%%%%%%%%%%%%%%%%%%%%%%%%%%%%%%%%%%%%%
\par
To produce the lift we use the decomposition (48) 
\begin{equation}  %%%%%%%   Begin  Eqn  (65)
\xi_{1} = 
\begin{pmatrix}
\alpha_{1} & 0 \\
0 & \alpha^{-1}_{1}
\end{pmatrix}
=
%%%%%
\begin{pmatrix}
1 & \alpha_{1} \\
0 & 1
\end{pmatrix}
%%%%%%
\begin{pmatrix}
1 & 0 \\
- \alpha^{-1}_{1} & 1
\end{pmatrix}
%%%%%%%
\begin{pmatrix}
1 & \alpha_{1} \\
0 & 1
\end{pmatrix}
%%%%%%%%
\begin{pmatrix}
0 & -1 \\
1 & 0
\end{pmatrix}
.
\end{equation}    %%%%%%%   End  Eqn  (65)
We suppose the homo-morphism $j_{1}$ to be an epi-morphism. To produce the lift one replaces in this formula  the elements $\alpha_{1}$, resp.  $\alpha_{1}^{-1}$,
by some  elements  $\tilde{\alpha}_{1}$, resp. $\tilde{\beta}_{1} \in \Lambda_{1, \mu}$ such that
$j_{1}(\tilde{\alpha}_{1}) = \alpha_{1}$, resp.  $j_{1}(\tilde{\beta}_{1}) =  \alpha_{1}^{-1}.$ 
The obtained element is
\begin{equation}  %%%%%%%   Begin  Eqn  (66)
\tilde{\xi}_{1} 
=
%%%%%
\begin{pmatrix}
1 & \tilde{\alpha}_{1} \\
0 & 1
\end{pmatrix}
%%%%%%
\begin{pmatrix}
1 & 0 \\
- \tilde{\beta}_{1} & 1
\end{pmatrix}
%%%%%%%
\begin{pmatrix}
1 & \tilde{\alpha}_{1} \\
0 & 1
\end{pmatrix}
%%%%%%%%
\begin{pmatrix}
0 & -1 \\
1 & 0
\end{pmatrix}
.
\end{equation}    %%%%%%%   End  Eqn  (66)
We use formula (66) and property (52) of elementary matrices to find the inverse of $\tilde{\xi}_{1}$
%%%%%%%%%
\begin{equation}  %%%%%%%   Begin  Eqn  (67)
\tilde{\xi}_{1}^{-1} 
=
%%%%%
\begin{pmatrix}
0 & 1 \\
-1 & 0
\end{pmatrix}
%%%%%%%
\begin{pmatrix}
1 & - \tilde{\alpha}_{1} \\
0 & 1
\end{pmatrix}
%%%%%%%%
\begin{pmatrix}
1 & 0 \\
 \tilde{\beta}_{1} & 1
\end{pmatrix}
%%%%%%%%
\begin{pmatrix}
1 & - \tilde{\alpha}_{1} \\
0 & 1
\end{pmatrix}
%%%%%%
.
\end{equation}    %%%%%%%   End  Eqn  (67)

%%%%%%%%%
Formulas (66) and (67) show that the invertible elements  $\tilde{\xi}_{1}$ and 
 $\tilde{\xi}_{1}^{-1}$ belong to $\Lambda_{1, \mu}$.  In addition, both elements satisfy 
 $j_{1}(\tilde{\xi}_{1}) = \xi_{1}$ and 
$ j_{1}(\tilde{\xi}_{1}^{-1}) =  (j_{1} (\tilde{\xi}_{1}))^{-1}   =   \xi_{1}^{-1}$.
At this point, with these two lifted elements we are entitled to produce the element 
\begin{equation}   %%%%%%%   Begin  Eqn  (68)
\tilde{U}_{1} =
\begin{pmatrix}
\tilde{\xi}_{1} &  0 \\
0 &  \tilde{\xi}_{1}^{-1}
\end{pmatrix}
\in \mathbb{O}_{4n}(\Lambda_{1, \mu}).
\end{equation}  %%%%%%%   End  Eqn  (68)
From this relation we obtain
\begin{equation}
j_{1} (\tilde{U}_{1}) =
\begin{pmatrix} 
\xi_{1} & 0\\
0 & \xi_{1}^{-1}
\end{pmatrix}
\sim_{l}  \; 2 \xi_{1}.
\end{equation}
We have proved the following result.
%%%%%%%%%%%%%%%%%%%%%%%%%%%%%
\begin{lemma}   %%%%%   Lemma  38   Begin
Suppose the homo-morphism  $j_{1}$ is epimorphic. 
\par
-1) 
Then  there exists a recepe which for any element  $\xi_{1} \in \mathbb{O}_{2n}(\Lambda^{'}_{\mu})$  produces an invertible 
element $\tilde{U}_{1}  \in  \mathbb{O}_{4n}(\Lambda_{1, \mu})$ such that
\begin{equation}  %%%%%%   Eqn   (70)   Begin
j_{1} (\tilde{U}_{1}) =
\begin{pmatrix} 
\xi_{1} & 0\\
0 & \xi_{1}^{-1}
\end{pmatrix}
\sim_{l}  \; 2 \xi_{1}.
\end{equation}   %%%%%%   Eqn  (70)  End
\par
-2)
Then $j_{1 \ast}(u_{1} + \xi_{1})$, from formula (63), lifts in $K_{1}(\Lambda_{1, \mu}) \otimes \mathbb{Z}[\frac{1}{2}]$ to $u_{1} + \frac{1}{2} \tilde{U}$ and 
\begin{equation}   %%%%  Eqn  (71)  Begin
[ u_{1} + \frac{1}{2} \tilde{U} ] = [u_{1}]  \in K_{1}(\Lambda_{1, \mu}) \otimes \mathbb{Z}[\frac{1}{2}]
\end{equation}  %%%%%   Eqn  (71)   End
\end{lemma}  %%%%%   Lemma  38   End
%%%%%%%%%%%%%%%%%%%%%%%%%%%%%%%%%

The properties -1) and -2) above being satisfied, the proof of the Theorem 36  proceeds as in the case of Lemma 29.
\end{proof}  %%%%%  End  Proof  Theorem  36

%%%%%%%%%%%%%%%%%%%%%%%%%%%%%%%  Section   8   End

\section{ $K^{loc}_{1}(\mathit{A})$  vrs.  $K_{1}(\mathit{A})$. }   %%%%%%%   Sect 9
%%%%%%%%%%%%%%%%%%%%%%%%%%%%%%%%%%%%%%%%%%%%%%%%%%%
\par
The following identities are well known and used as building blocks of $K$-theory, see  \cite{Whitehead},  \cite{Milnor},  \cite{Karoubi}, \cite{Blackadar}, \cite{Rosenberg}, \cite{Rordam_Lansen_Lautsen}, \cite{Weibel}.
%%%%%%%%%%%%%%%%%%%%%%%%%%%%
\par
To facilitate the reading of this article we summarise some basic facts from the classical algebraic $K$-theory and put them in prospective with some  considerations discussed here.
\begin{definition}  Whitehead Group $K_{1}(\mathit{A})$.    %%%%   Definition  39 Begin
\par
By definition
\begin{equation}   %%%%%%%   Eqn  (72)   Begin
K_{1}(\mathit{A}) := \injlim_{n \in \mathbb{N}}  \mathbb{GL}_{n} /  [ \mathbb{GL}_{n} ,\mathbb{GL}_{n}]
\end{equation}     %%%%%%%   Eqn  (72)     End
The group structure in $K_{1}(\mathit{A})$ is given by matrix multiplication
\begin{equation}  %%%%%%%%%%%   Eqn  (73)   Begin   
[ A ] + [B] :=
[ \;A.B\;].
\end{equation}  %%%%%%%   Eqn  (73)     End
\end{definition}  %%%%%%%  %%%%%%%%%%         Definition  39    End
\begin{theorem}   %%%%%%                                              Begin   Theorem 40
-i) Any commutator is stably isomorphic to a product of elements of the form $\mathcal{O}_{2n}(\mathit{A})$. More specifically, for any $ A , B \in \mathbb{GL} (\mathit{A}_{\mu})$
\begin{equation}     %%%%%   Eqn   (74)   Begin
\begin{pmatrix}
ABA^{-1} B^{-1} & 0\\
0  &  1
\end{pmatrix}
=
%%%%%%%%%%
\begin{pmatrix}
A & 0\\
0  &  A^{-1}
\end{pmatrix}
%%%%%%%%%%
\begin{pmatrix}
B& 0\\
0  &  B^{-1}
\end{pmatrix}
%%%%%%%%%%
\begin{pmatrix}
(BA)^{-1} & 0\\
0  &  BA
\end{pmatrix} .
\end{equation}     %%%%%%%%%%%   End  Eqn  (74)

-ii) If    $ A \in \mathbb{GL}_{n}(\mathit{A}_{\mu})$, then  (this repeats  Definition 15 formula (12) )
\begin{equation}  %%%%%%%%   Begin  Eqn  (75)
\mathcal{O}_{2n} (A) :=
\begin{pmatrix}
A & 0 \\
0  & A^{-1}
\end{pmatrix}
=
%%%%%
\begin{pmatrix}
1 & A \\
0  & 1
\end{pmatrix}
%%%%%%
\begin{pmatrix}
1 & 0 \\
-A^{-1}  & 1
\end{pmatrix}
%%%%%
\begin{pmatrix}
1 & A \\
0  & 1
\end{pmatrix}
%%%%%%
\begin{pmatrix}
0 & -1 \\
1  &  0
\end{pmatrix}
%%%%%
\end{equation}   %%%%%%%%   End   Eqn  (75)
-iii) Any elementary matrix is a commutator  (this repeats formula (53) )
\begin{equation}  %%%%%%%   Begin  Eqn  (76)
E_{ij} (A) = [ E_{ik}(A), E_{kj}(1)],   \hspace{0.2cm}  for \; any\; i, j, k \; distinct \; indices.
\end{equation}  %%%%%%  End  Eqn  (76)
%%%%
-iv) For any $A, B  \in \mathbb{GL}_{n}( \mathit{A}_{\mu})$, $A+B$ is stably equivalent to $AB$ and $BA$ modulo (multiplicatively) elements of the form $\mathcal{O}_{2n}(\mathit{A})$
\begin{equation}   %%%%   Begin  Eqn  (76)
\begin{pmatrix}
A  &  0\\
0  &  B
\end{pmatrix}
%%%%%%
=
%%%%%%
\begin{pmatrix}
AB & 0\\
0  &  1
\end{pmatrix}  
%%%%%%%    
\begin{pmatrix}
B^{-1} & 0\\
0  &  B
\end{pmatrix} 
=
\begin{pmatrix}
B^{-1} & 0\\
0  &  B
\end{pmatrix} 
\begin{pmatrix}
BA & 0\\
0  &  1
\end{pmatrix}  
. 
\end{equation}  %%%%%%   End  Eqn  (76)
\par
%%%%%%%%%%%
-v) For any $ A, \; B \in \mathbb{GL}_{n} ( \mathit{A}_{\mu}) $ one has the identity
\begin{equation}    %%%%%   Begin  Eqn  (77)
\begin{pmatrix}
ABA^{-1}  &  0 \\
       0         &  1
\end{pmatrix}
%%%%%%
=
%%%%
\begin{pmatrix}
A   &   0  \\
0   &   A^{-1}
\end{pmatrix}
%%%%%%
%%%%
\begin{pmatrix}
B  &   0 \\
0  &    1
\end{pmatrix}
%%%%%%
%%%%
\begin{pmatrix}
A^{-1}  &  0  \\
    0      &   A
\end{pmatrix} = 
%%%%%%
%%%%%%%%%%
\begin{pmatrix}
A & 0\\
0  &  A^{-1}
\end{pmatrix}
%%%%%%%%%%
\begin{pmatrix}
B& 0\\
0  &  1
\end{pmatrix}
%%%%%%%%%%
\begin{pmatrix}
A & 0\\
0  &  A^{-1}
\end{pmatrix}^{-1}
. 
\end{equation}    %%%%%%%   End  Eqn   (77)
\par
\end{theorem}   %%%%%%%%   Theorem   40   End

%%%%%%%%%%%%%%%%%%%%%%%%%%%%%%%%%%%%%%%%%%%%%%%%%%%%%%%
\begin{proof}   %%%%%%   Begin  Proof  Theorem  40
The proof is obtained by direct verification.
\end{proof}   %%%%%%%%   End  Proof  Theorem 40
%%%%%%%%%%%%%%%%%%%%%%%%%%%%%
\begin{theorem} (\cite{Whitehead}, \cite{Bass}, \cite{Swan}, \cite{Milnor}, \cite{Karoubi}, \cite{Rosenberg} )   %%  Theorem 41  Begin
\par
-i)  $[\mathbb{GL}(\mathit{A}),  \mathbb{GL}_{n}(\mathit{A}) ]  =   \mathbb{E}_{n}(\mathit{A})$
\par
-ii) $[A B A^{-1}B^{-1}]  = [E_{ij}(A) ] = 0  \in K_{1} ( \mathit{A})$   %%%%%%   Eqn  () 
\par
-iii) $[A ]+ [B]  = [ \;A B\; ] = [ \;BA\; ]  = [ B ]+ [ A] \in  K_{1} (\mathit{A}) $. 
\par
Therefore, $\mathbb{GL}(\mathit{A})$ is an Abelian group.
\par
-iv) $[\mathcal{O}_{2n} (\mathit{A})]  =  [1_{n} ]   =  1  \in K_{1}(\mathit{A})$.
\par
$- [A] = [A^{-1}] $  for any $A \in \mathbb{GL}(\mathit{A})$.
\par
-v) 
$[A B A^{-1}]  = [ B ]   \in K_{1} ( \mathit{A}) $.   %%%%%%   Eqn  () 
\end{theorem}   %%%%%%%%%%%%%%%%%                      Theorem  41    End

%%%%%%%%%%%%%%%%%%%%%%%
\begin{proof} %%%%%%%%%                                            Proof  Theorem  41     Begin
Relation (74) says that any commutator is a product of matrices of type $\mathcal{O}_{n}(\mathit{A})$. Formula (75) says that any matrix of type  $\mathcal{O}_{n}(\mathit{A})$ is a product of elementary matrices. Viceversa, (76) says that any elementary matrix is a commutator. This proves -i).
\par
-ii) follows from the definition of $K_{1}(\mathit{A})$.
\par
-iii) Formula (77) implies the first part of -iii).  Formula (75) combined with Theorem 41. -i) and the definition of 
$K_{1}(\mathit{A})$ imply that $[ \mathcal{O}(A)] = 0 \in  K_{1}(\mathit{A})$, which proves -iv). Formula (77) shows that 
$[A \oplus B] = [A] .[B]$, which completes the proof of -iii).
\par
-v) This follows from formula (77) together with -iii) and -ii).
 \end{proof}    %%%%%%%                                                       Proof  Theorem  41   End
 %%%%%%%%%%%%%%%%%%%%%%%%%%%%%%%%%%%%%%%%%%%%%%%%%%%%%%%%%
\begin{remark}   %%%%%%%%                  Remark  42   Begin
-i) The construction of  the classical $K_{1}(\mathit{A})$ does not require to use  the Grothendieck completion. This is because the factorisation
of   $\mathbb{GL}(\mathit{A}  )$ through  the commutator sub-group insures property -iv), which provides the group structure.
\par
{\em The group structure in $ K_{1}(\mathit{A})$ is essentially due to the fact that the both the direct sum and product of invertible matrices define the addition/multiplication in $K_{1}(\mathit{A})$. These properties imply that  
$\mathbb{O}_{n}(A)$ represents the zero class  in $ K_{1}(\mathit{A})$, which insures existence of opposite elements.}
\end{remark}   %%%%%   Remark  42  End
\par
The next remark refers to the specific parts of the construction of $K_{1}^{loc}$ which make our construction different from the classical one.
\begin{remark} %%%%%   Remark   43   Begin
\par
-i) In our construction the commutator sub-group and the group generated by elementary matrices are avoided because the number of multiplications needed to generate these sub-groups might be un-bounded. Any multiplication increases the size of the support (in the definition of localised algebras this means that the support of the elements could not be controlled).
\par
For this reason, in the whole article we don't use more than three multiplication of elements in the algebra. The multiplication of elements is replaced by sums and direct sums.  The increase of the of the size of matrices replaces in our construction the need to perform multiple multiplications. In our definition of $K_{1}^{loc}(\mathit{A})$, 
where multiplications could not be avoided, the corresponding increase in the size of the supports  is absorbed by the projective limit. 
\par
-ii) Our construction uses the elements $\mathcal{O}(\mathit{A})$ {\em additively} and not {\em multiplicatively}. In the classical
construction of $K_{1}(\mathit{A})$,  the elements of $\mathcal{O}(\mathit{A})$ are used multiplicatively. Indeed, formula (74) shows that the commutator
sub-group is generated by these elements. Unfortunately, to generate the commutator sub-group it is necessary to perform an {\em un-bounded number of multiplications}; in our construction an un-bounded number of multiplications is not allowed. 
\par
-iii) The construction of $K^{loc}_{1}(\mathit{A})$ uses the factorisation of $\mathbb{GL}(\mathit{A})$ through the smaller sub-group of inner auto-morphisms. The class of an invertible element $u$ modulo inner auto-morphisms, which we called abstract Jordan form, denoted $J(u)$,  contains more information than the class of the invertible element modulo the commutator sub-group.
In the classical $K_{1}(\mathit{A})$ the elements of $\mathcal{O}_{n}(\mathit{A})$ represent the null element. 
\par
If $\mathit{A}$ were the algebra of complex matrices $\mathbb{M}(\mathbb{C})$ and $u$ belonged to this algebra, then
$J(u)$ could be identified with the Jordan canonical form of the matrix $u$. The Jordan canonical form of the
matrix $u \oplus u^{-1}$ is precisely $J(u) \cup J(u^{-1})$ modulo permutations of the Jordan blocks. 
It is clear that $J( u \oplus u^{-1}) $ could never be conjugate to the the identity element unless $u = 1_{n}$.
 In general, the abstract Jordan canonical form  of $u \oplus u^{-1}$ could not be conjugate to the identity unless $u$ itself is the identity. 
 \par
 Given that the elements $\mathcal{O}(\mathit{A})$ represent the null element and 
play multiple roles in the $K$-theory (see above discussion), we find it natural, for the definition of  $K_{1}^{loc}(\mathit{A})$, to 
quotient $\mathbb{GL}(\mathit{A})/ \sim_{sl}$ through the {\em additive} sub-group generated by direct sums of elements in
$\mathcal{O}(\mathit{A})$. In other words, by this factorisation we decree 
that the Jordan canonical forms of the elements $u$ and  $u^{-1}$ are opposite one to each other.
The {\em additive} group generated by
elements $\mathcal{O}(\mathit{A})$ is contained in the commutator sub-group; this property insures the fact that there exists a natural epi-morphism from    $K_{1}^{loc}(\mathit{A})$ to  $K_{1}(\mathit{A})$.
\par
-iv) Additionally, the projective limit (Alexander-Spanier type construction, made possible by the filtration $\mathit{A}_{\mu}$ of the algebra $\mathit{A}$) makes the algebraic $K^{loc}_{i}$-theory reacher than the classical $K_{i}$-theory, $i = 0,  1$.  
\end{remark}  %%%%%%%%                     Remark  43    End
%%%%%%%%%%%%%%%%%%%%%%%%%%%%%%%%%%%%%%%%%%%%%%%%%%%%%%%%%

%%%%%%%%%%%%%%%%%%%%%%%%%%%%%%%%%%
\begin{theorem}    %%%%%%   Theorem  44   Begin
-i) There is a canonical epi-morphism
\begin{gather}   %%%%%   Eqn  (79)   Begin
\Pi:  K^{loc}_{1} (\mathit{A})  \longrightarrow   K_{1} (\mathit{A}) \\
Ker \; \Pi =   [\; \mathbb{GL}(\mathit{A}),  \mathbb{GL}(\mathit{A}) \;] \; /  \;Inner(\mathit{A}).
\end{gather}   %%%%%   Eqn  (80)   End
\end{theorem}  %%%%%%   Theorem  44   End
%%%%%%%%%%%%%%%%%%%%%%%%%%%%%%%%%%%   Section 9  End
%%%%%%%%%%%%%%%%%%%%%%%%%%%%%%%%%%%%  Section   10
\section{Connecting homo-morphism $\partial:  K_{1}^{loc} (\Lambda^{'}) \longrightarrow K_{0}^{loc}(\Lambda)$.}  %%%%   Sect. 
In this section we assume that the diagram (6) satisfies Hypotheses 1, 2. 3.
\subsection{Connecting homo-morphism }   %%%%%%%   Sect.   10.1
%%%%%%%%%%%%%%%%%%%%%%%%%%%%%%%%%                        
\begin{definition}  {\em Connecting homomorphism-first form.}   %%%%  Definition 45  Begin    \partial
\par
Let $U \in \mathbb{GL}_{n}(\Lambda^{'}_{\mu})$. Theorem 27 associates the idempotent $p = p(1_{n}, 1_{n}, U)$.
By definition, $\partial[U] = [p(1_{n}, 1_{n}, U)]$.
\end{definition}   %%%%                                          Definition 45    End
\par
$\partial [U]$ may be produced explicitely in a different manner.
So, let $[U] \in K_{1}^{loc}(\Lambda^{'})$ be such that $U \in \mathbb{GL}_{n}(\Lambda^{'}_{\mu})$ for some 
$n$ and $\mu$.
\par
Let $A, \; B \in  \mathbb{GL}_{n} (\Lambda_{1, \mu})$  be liftings of $U$, resp. $U^{-1}$, in 
$\mathbb{M}_{n} (\Lambda_{1, \mu})$. Such liftings exist because we assume $j_{1}$ is surjective (note the change of the index). Here we do not assume that $A$ and $B$ are invertible.
\par
With $A$ and $B$ one associates $S_{0}= 1 -  BA$ and $S_{1}= 1 -  AB$; these elements belong to
$\mathbb{M}_{n} (\Lambda_{1, \mu -1 })$.
The matrices $S_{0}$, $S_{1}$ satisfy 
\begin{equation}  %%%%%%%   Eqn  (81)  Begin
j_{2, \ast } (S_{0}) = j_{1, \ast } (S_{1}) = 0.
\end{equation}  %%%%%%%   Eqn  (81)   End
With these matrices one associates the invertible matrix    
(ref. \cite{Connes_Moscovici})
\begin{gather}   %%%%%%   Eqn  (82)  Begin
L =
\begin{pmatrix}
S_{0}  &  - (1 + S_{0}) B  \\
A  &  S_{1}
\end{pmatrix}
\in \mathbb{GL}_{2n}(\Lambda_{1, \mu - 2});
\end{gather}   %%%%%%   Eqn  (82)  End
the inverse of the matrix $L$ is 
\begin{gather}  %%%%%%   Eqn  (83)  Begin
L^{-1} =
\begin{pmatrix}
S_{0}  &   (1 + S_{0}) B  \\
- A  &  S_{1}
\end{pmatrix}
\in \mathbb{GL}_{2n}(\Lambda_{1, \mu - 2}).
\end{gather}    %%%%%%   Eqn  (83)   End
Let $e_{1},  e_{2} $ be the idempotents

\begin{gather}  %%%%%%   Eqn  (84)  Begin
e_{1} =
\begin{pmatrix}
1 &   0  \\
0  &  0
\end{pmatrix}, \;\;\; e_{2} =
\begin{pmatrix}
0 &   0  \\
0  &  1
\end{pmatrix}
\;\; \in \mathit{Idemp}_{2n}(\Lambda_{1, \mu }).
\end{gather}    %%%%%%   Eqn  (84)  End

The invertible matrix $L$ is used to produce the idempotent

\begin{gather}    %%%%%%   Eqn  (85)  Begin
P  :=  L e_{1} L^{-1} =
\begin{pmatrix}
S_{0}^{2}  &   S_{0} (1 + S_{0}) B  \\
S_{1}A  &  1 - S_{1}^{2}
\end{pmatrix}
\in \mathit{Idemp}_{2n}(\Lambda_{1, \mu - 2}).
\end{gather}    %%%%%%   Eqn  (85)  End
\par
A matrix whose entries belong to the double ring $\Lambda$ will be called {\em double-matrix}.
\par
The idempotent $P$ is used to construct the double-matrix idempotent
\par
\begin{gather}  %%%%%%   Eqn  (86)  Begin
\mathbb{P}_{U}  :=
\begin{pmatrix}
(S_{0}^{2}, \; 0)  &   (S_{0} (1 + S_{0}) B, \; 0)  \\
(S_{1}A, \; 0)   &  (1 - S_{1}^{2}, \; 1 )
\end{pmatrix}
\in \mathit{Idemp}_{2n}(\Lambda_{1, \mu - 2} \oplus    \Lambda_{2, \mu - 2}).
\end{gather}   %%%%%%   Eqn  (86)  End
The matrix $\mathbb{P}_{U}$ is an idempotent in $\mathbb{M}_{2n}(\Lambda_{\mu - 2})$.
Indeed, using (84) one gets 
$(j_{1, \ast} - j_{2, \ast}) \mathbb{P}_{U} = 0$ and therefore
$\mathbb{P}_{U} \in \mathbb{M}_{2n}(\Lambda_{\mu - 2})$.
The fact that the matrix $\mathbb{P}_{U}$ is an idempotent in $\mathbb{M}_{2n, \mu -2}(\Lambda)$ follows from the fact that  $\mathit{P} $ and $e_{2}$  are idempotents along with Remark 8, \S 5. 
%%%%%%%%%%%%%%%%%%%%%%%%%%%
\begin{definition} {\em Connecting homomorphism-second form.}   %%   Definition 46  \partial  Begin
 (compare \cite{Connes_Moscovici} for non-localised $K$-theory).   
\par
For any $[U] \in K_{1}^{loc} (\Lambda^{'})$ one defines the {\em connecting homomorphism} 
$\partial : K_{1}^{loc}(\Lambda') \longrightarrow K_{0}^{loc}(\Lambda)$ by
\begin{equation}   %%%%  Eqn (87)   Begin
\partial [U] := [ \mathbb{P}_{U}]  - [(e_{2}, e_{2})] \in  K_{0}^{loc} (\Lambda).
\end{equation}   %%%%%   Eqn  (87)   End
\end{definition}   %% %%%%%%%%%%%                                Definition 46              \partial   -  End
\par
%%%%%%%%%%%%%%%%%%%%%%%%%%%%% 
\subsection{Extension of the Connecting Homomorphism $\partial$.}    %%%   Begin  \subsection  8.2

In this paper we will need an extension of the formula for the definition of $\partial$. In the proof of the exactness, Theorem 51.{\bf -iii)},  we will need to know how the connecting homomorphism, formula (86), behaves with respect to stabilisations. 
For this purpose we consider a more general situation than that considered in the \S 10.1.
\par
Let  $U \in \mathbb{GL}_{m+n} (\Lambda')$ and 
\begin{equation}   %%%%%   Eqn  (88)   Begin
e_{(0,n)} =
\begin{pmatrix}
0 & 0 \\
0 & 1_{n}
\end{pmatrix} 
\in \mathbb{M}_{m+n} (\Lambda')
\end{equation}    %%%%%   Eqn  (88)   End
be such that the diagram
%%%%%%%%%%%%%%%%%%%%%%%
\begin{equation}  %%%%%   Eqn  (89)   Begin
\begin{CD}
\Lambda^{'m+n}  @  >  {  U }   >>   \Lambda^{'m+n}  \\
@VV  { e_{(0,n)} }  V   @    VV { e_{(0,n)} }  V \\
\Lambda^{'m+n}  @  > { U  } >>   \Lambda^{'m+n} 
\end{CD}
\end{equation}   %%%%%   Eqn  (89)   End
%%%%%%%%%%%%%%%%%%%%
is commutative. This condition is used at the end of the construction of $\partial [U]$; it is needed to
show that $\mathbb{P}_{U} \in \mathbb{I}demp_{m+n}(\Lambda)$, formula (97).   
\par
The case discussed in \S 10.1  corresponds  in this subsection to $m = 0$ .  
\par
%%%%%%%%%%%%%%%%%%%%%%%%%%%%%%%%%%%%%%%
We proceed as before.
Let $A, \; B \in  \mathbb{GL}_{m+n} (\Lambda_{1, \mu})$  be liftings of $U$, resp. $U^{-1}$, in 
$\mathbb{M}_{m+n, \mu} (\Lambda_{1})$. Such liftings exist because we assume $j_{1}$ is surjective.
\par
With $A$ and $B$ one associates $S_{0}= 1 -  BA \in \mathbb{GL}_{m+n} (\Lambda_{1, \mu})$ and 
$S_{1}= 1 -  AB \in \mathbb{GL}_{m+n} (\Lambda_{1, \mu})$.
The matrices $S_{0}$, $S_{1}$ satisfy 
\begin{equation}   %%%%%   Eqn  (90)   Begin
j_{1, \ast } (S_{0}) = j_{1, \ast } (S_{1}) = 0.
\end{equation}  %%%%%   Eqn  (90)   End
With these matrices one associates the invertible matrix    
(ref. \cite{Connes_Moscovici})
\begin{gather}  %%%%%   Eqn  (91)   Begin
L =
\begin{pmatrix}
S_{0}  &  - (1 + S_{0}) B  \\
A  &  S_{1}
\end{pmatrix}
\in \mathbb{GL}_{2(m+n)}(\Lambda_{1, \mu - 2});
\end{gather}   %%%%%   Eqn  (91)   End
the inverse of the matrix $L$ is 
\begin{gather}  %%%%%   Eqn  (92)   Begin
L^{-1} =
\begin{pmatrix}
S_{0}  &   (1 + S_{0}) B  \\
- A  &  S_{1}
\end{pmatrix}
\in \mathbb{GL}_{2(m+n)}(\Lambda_{1, \mu - 2}).
\end{gather}  %%%%%   Eqn  (92)   End
%%%%%%%%%%%%%%%%%%%%%%%%%%%%%%%  Begin New      \partial
Let $e_{1} $ be the idempotent
\begin{gather}  %%%%%   Eqn  (93)   Begin
e_{1} = 
\begin{pmatrix}
e_{(0,n)} & 0\\
0 & 0
\end{pmatrix} \;\; \in \mathbb{GL}_{2(m+n)}(\Lambda_{1, \mu }).
\end{gather}   %%%%%   Eqn  (93)   End
and

\begin{gather}  %%%%%   Eqn  (94)   Begin
e_{2} = 
\begin{pmatrix}
0 & 0\\
0 & e_{(0,n)}
\end{pmatrix} \;\; \in \mathbb{GL}_{2(m+n)}(\Lambda_{1, \mu }).
\end{gather}   %%%%%   Eqn  (94)   End

The invertible matrix $L$ is used to produce the idempotent

\begin{gather}  %%%%%   Eqn  (95)   Begin
P  :=  L e_{1} L^{-1} =
\begin{pmatrix}
S_{0} \;e_{(0,n)}  \; S_{0} &   S_{0} \; e_{(0,n)}  \; (1 + S_{0}) B  \\
A \;e_{(0,n)}  \;S_{0} &  A \;e_{(0,n)}  \; (1 + S_{0}) B
\end{pmatrix}
\in \mathit{Idemp}_{2(m+n)}(\Lambda_{1, \mu - 2}).
\end{gather}   %%%%%   Eqn  (95)   End

The idempotent $P$ is used to construct the double-matrix idempotent

\begin{gather}   %%%%%   Eqn  (96)   Begin
\mathbb{P}_{U}  :=
\begin{pmatrix}
(S_{0}  \;e_{(0,n)}  \; S_{0}, \; 0)  &   (S_{0} \; e_{(0,n)}  \;(1 + S_{0}) B, \; 0)  \\
(A \; e_{(0,n)}  \;S_{0}, \; 0)   & (A \; e_{(0,n)}  \;  (1 - S_{0}) B, \; e_{(0,n)} )
\end{pmatrix}
\in \mathbb{I}demp_{2(m+n)}(\Lambda_{1, \mu - 2} \oplus    \Lambda_{2, \mu - 2}).
\end{gather}    %%%%%   Eqn  (96)   End
The matrix   $\mathbb{P}_{U}$  is an idempotent in $\mathbb{M}_{2(m+n)}(\Lambda_{\mu - 2})$.
We may verify directly that  $(j_{1, \ast} - j_{2, \ast}) \mathbb{P}_{U} = 0$. Indeed, $j_{1 \ast}$ is a ring homomorphism
and  $j_{1 \ast} (S_{0}) = j_{1 \ast} (S_{0}) = 0$; finally,  the hypothesis (89) gives
\begin{equation} %%%%%%   Eqn   (97)   Begin
j_{1}(A \; e_{(0,n)}  \;  (1 - S_{0}) B) = U e_{(0,n)} U^{-1} = e_{(0,n)} = j_{2} (e_{(0,n)} ).
\end{equation}   %%%%%   Eqn  (97)   End
 Therefore
$\mathbb{P}_{U} \in \mathbb{M}_{2n}(\Lambda_{\mu - 2})$.
The fact that the matrix $\mathbb{P}_{U}$ is an idempotent in $\mathbb{M}_{2(m+n)}(\Lambda_{\mu -2})$ follows from the fact that $\mathit{P} $ and $e_{2}$  are idempotents along with the discussion above. 
%%%%%%%%%%%%%%%%%%%%%%%%%%%
\begin{definition}    {\em Connecting homomorphism - third form.} %%   Definition 47  (Extension \partial)  Begin
\par
We suppose the assumptions and constructions of \S 10.2 above are in place. 
For any $[U] \in K_{1}^{loc} (\Lambda^{'})$ one defines the {\em connecting homomorphism} 
$\partial : K_{1}^{loc}(\Lambda') \longrightarrow K_{0}^{loc}(\Lambda)$ by
\begin{gather}   %%%%   Eqn (98)   Begin
\partial [U] := [ \mathbb{P}_{U}]  - [(e_{2}, e_{2})] = 
\begin{pmatrix}
(0 , \; 0)  &   (0, \; 0)  \\
(0, \; 0)   & (A e_{(0,n)}  B, \; e_{(0,n)}  )
\end{pmatrix} 
 - [(e_{2}, e_{2})]
\in  K_{0}^{loc} (\Lambda).
\end{gather}  %%%%   Eqn (98)  End
\end{definition}   %%%%%%%                                        Definition 47 (Extension  \partial)   -  End
%%%%%%%%%%%%%%%%%%%%%%%%%%%%%%%%%%%% 
\begin{remark}  %%%%%%   Remark   48  Begin
If the lifts $A$, resp. $B$, of the isomorphisms $U$, resp. $U^{-1}$, in $\Lambda_{1}^{m+n}$ are inverse one to each other,  then the corresponding matrices $S_{0} = S_{1} = 0$ and 
\begin{equation}   %%%%   Eqn   (99)    Begin
\partial [U] := [ \mathbb{P}_{U}]  - [(e_{2}, e_{2})] \in  K_{0}^{loc} (\Lambda).
\end{equation}   %%%%   Eqn (99)   End
\end{remark}   %%%%%%   Remark   48   End
%%%%%%%%%%%%%%%%%%%%%%%%%%%%%%%%%%%%%

\begin{proposition} %%%%%   Proposition 49   -  Begin 
The connecting homomorphism $\partial $  is well defined.
\end{proposition}  %%%%%   Proposition  49  - End

\begin{proof}     %%%%%   Proof  Proposition  49   Begin
It is easy to see that $\mathbb{P}_{U}$ behaves correctly with respect to stabilisation and sums.
\par
Given that any $\partial \xi$, with $\xi \in \mathbb{O}_{\mu}(\Lambda^{'}{\mu})$, represents the zero element in $K_{1}(\mathit{A}_{\mu})$, see Proposition 22. -ii), 
we have to show that $\partial \xi = 0$. To do this we take into account the 
Definition  46 together with formula (48).  In fact, this last formula shows that the invertible element $\xi$ has an invertible lifting in 
$\mathbb{GL}(\Lambda_{2, \mu})$. This lifted element together with its inverse produce a pair of elements $A$, $B$, see   Definition
45, for which the corresponding elements $S_{0} = S_{1}  = 0.$ Formulas (86), (87) show that $\partial \xi= 0.$ 
\par
It remains to follow up how $\mathbb{P}_{U}$ depends of the choice of the lifts $A$ and $B$ of $U$ and $U^{-1}$. 
We observe that a different choice of $A$ and $B$ has the effect of modifying just the matrix $L$. We will show that if
$A'$ and $B'$ are two such different lifts and $L'$ is the corresponding matrix, then
\begin{equation}    %%%%%   Eqn  (100)   Begin
L' = \tilde{L} \;  L,     \hspace{0.5cm}  with   \hspace{0.2cm}  \tilde{L} \in \mathbb{GL}_{2(m+n)} (\Lambda)
\end{equation}   %%%%%   Eqn  (100)    End
and hence the corresponding idempotents $\mathbb{P}_{U} := L e_{1} L^{-1}$,
$\mathbb{P}'_{U} := L' e_{1} L'^{-1}$ are conjugate.
 \par
 To better organise the computation, we change the liftings one at the time. 
 \par
 We begin with $A$. Let $\tilde{A} = A + K$ with $j_{1}(K) = 0$. Let $\tilde{S}_{0} = 1 - B \tilde{A}$, 
 $\tilde{S}_{1} = 1 -  \tilde{A} B$,  $\tilde{L}$,  $\tilde{P}$ and $\tilde{\mathbb{P}}_{U}$ be the corresponding elements.
 A direct computation gives
  
 \begin{gather}    %%%%%%%   Eqn  (101)   Begin
 \tilde{L} L^{-1} = 
 \begin{pmatrix}
1 - B K &  - B K B  \\
     K   &  1 +  K B
\end{pmatrix}   
\end{gather}     %%%%%%%   Eqn  (101)  End
or

\begin{gather}   %%%%%%%   Eqn  (102)   Begin
 \tilde{L}  = 
 \begin{pmatrix}
1 - B K &  - B K B  \\
     K   &  1 +  K B
\end{pmatrix}    %%%%%%%   Eqn  (102)  End
 L.
\end{gather}

We know that $\tilde{L}$ and $L$ are invertible matrices; therefore the RHS of (101) is an invertible matrix.
\par
The corresponding idempotent $\tilde{P}$ is
\begin{gather}   %%%%%%%   Eqn  (103)   Begin
\tilde{P} = \tilde{L}  . e_{1} . \tilde{L}^{-1} =    
\begin{pmatrix}
1 - B K &  - B K B  \\
     K   &  1 +  K B
\end{pmatrix}
 L. e_{1} . L^{-1}
\begin{pmatrix}
1 - B K &  - B K B  \\
     K   &  1 +  K B
\end{pmatrix} 
^{-1} = \\        %%%%%%   Eqn (103)   End
= \begin{pmatrix}   %%%%%     Eqn   (104)  Begin
1 - B K &  - B K B  \\
     K   &  1 +  K B
\end{pmatrix}
 P
\begin{pmatrix}
1 - B K &  - B K B  \\
     K   &  1 +  K B
\end{pmatrix} 
^{-1}
\end{gather}   %%%%%%%    Eqn  (104)   End
and furthermore
\begin{gather}   %%%%%%  Eqn  (105)   Begin
\tilde{\mathbb{P}}_{U} =
(
\begin{pmatrix}
1 - B K &  - B K B  \\
     K   &  1 +  K B
\end{pmatrix}, \; 1) \;
 \mathbb{P}_{U} \;
( \begin{pmatrix}
1 - B K &  - B K B  \\
     K   &  1 +  K B
\end{pmatrix}, \; 1)
^{-1}.
\end{gather}   %%%%%%%   Eqn  (105)  End
Therefore, $[\tilde{\mathbb{P}}_{U}] = [\mathbb{P}_{U}] \in K_{0}^{loc}(\Lambda)$.
\par
It remains to see what happens if $A$ remains unchanged and the lifting B is changed. Let $\tilde{B} = B + H$, with
$j_{1} (H) = 0$.  Let $\tilde{\tilde{L}}$,  $\tilde{\tilde{P}}$ and  $\tilde{\tilde{\mathbb{P}}}_{U}$ be the corresponding
matrices.  A direct computation gives 

\begin{gather}   %%%%%   Eqn   (106)   Begin
\tilde{\tilde{L}} . L^{-1} =    
\begin{pmatrix}  
1 + \Delta_{11}  &  \Delta_{12}  \\
\Delta_{21}  &  1 + \Delta_{22}
\end{pmatrix}
  \in \mathbb{M}_{2n} ( \Lambda_{\mu -4}) 
\end{gather}    %%%%%%   Eqn(106)  End
where 
\begin{align*}
\Delta_{11}  =  &  HA - HAHA - BAHA     \\
\Delta_{12}  =  & -2H + HAB + HAH + BAH - BAHAB - HAHAB  \\
\Delta_{21}   =  & AHA    \\
\Delta_{22}  =   & - AH + AHAB.
\end{align*} 

The RHS of (106) is a product of invertible matrices; therefore, it is an invertible matrix. Proceeding as above we get
\begin{gather}   %%%%%   Eqn   (107)   Begin
\tilde{\tilde{\mathbb{P}}}_{U} =  
( \begin{pmatrix}  
1 + \Delta_{11}  &  \Delta_{12}  \\
\Delta_{21}  &  1 + \Delta_{22}
\end{pmatrix}
, \; 1)
\; \mathbb{P}_{U} \; 
( \begin{pmatrix}  
1 + \Delta_{11}  &  \Delta_{12}  \\
\Delta_{21}  &  1 + \Delta_{22}
\end{pmatrix},
\; 1)
^{-1},
\end{gather}  %%%%%   Eqn   (107)  End
which completes the proof of the Proposition 49.  %%%%  (End Proof  Prop. 49)
\end{proof}   %%%%%%   Proof  Proposition  49  End
 
%%%%%%%%%%%%%%%%%%%%%%%%%%%%%%%%%%
\begin{remark}   %%%%  Remark  50  Begin
The connecting homomorphism $\partial$ could be defined by the same formula (95), (96), (98) where 
$L \in \mathbb{GL}_{2n}(\Lambda_{1, \mu})$ is {\em any} invertible lifting of the matrix
\begin{gather}   %%%%   Eqn  (108)  Begin
\begin{pmatrix} 
0 & -U^{-1}\\
U & 0
\end{pmatrix}.
\end{gather}   %%%%%   Eqn (108)  End
If $L$, $L'$ are two such liftings, the corresponding idempotents $\mathbb{P}_{U}$ are conjugate
\begin{equation}  %%%%   Eqn  (109)  Begin
\mathbb{P}'_{U}    =    (L'  L^{-1}) \mathbb{P}_{U}    (L'  L^{-1})^{-1}.
\end{equation}  %%%%   Eqn  (109)  End
\end{remark}  %%%%%   Remark  50  End
%%%%%%%%%%%%%%%%%%%%%%%%%%%%%%%%%%   Sect.  10   $\partial$      END

\section{Six terms exact sequence.}     %%%%%%%%%%    Sect.   11    Six Term Exact Sequence  BEGIN
\begin{theorem} %%%%%%%  Theorem  51  -  Begin
The six terms sequence 
\begin{gather}   %%%%   Eqn  (110) (111) -  Begin
K_{1}^{loc} (\Lambda) \otimes \mathbb{Z}[\frac{1}{2}] \longrightarrow
K_{1}^{loc} (\Lambda_{1})  \otimes \mathbb{Z}[\frac{1}{2}] \oplus K_{1}^{loc} (\Lambda_{2}) \otimes \mathbb{Z}[\frac{1}{2}]   \overset{ j_{1 \ast} - j_{2 \ast} }{\longrightarrow} 
K_{1}^{loc} (\Lambda^{'})  \otimes \mathbb{Z}[\frac{1}{2}]  \overset{\partial}{\longrightarrow} \\
K_{0}^{loc} (\Lambda) \otimes \mathbb{Z}[\frac{1}{2}]   \overset{ (i_{1 \ast}, i_{2 \ast}) }{\longrightarrow} 
K_{0}^{loc} (\Lambda_{1}) \otimes \mathbb{Z}[\frac{1}{2}]   \oplus K_{0}^{loc} (\Lambda_{2}) \otimes \mathbb{Z}[\frac{1}{2}] 
\longrightarrow  K_{0}^{loc}(\Lambda^{'}) \otimes \mathbb{Z}[\frac{1}{2}] .
\end{gather}  %%%%   Eqn  (110) (111)   -  End
is exact. 
\par
More precisely, the following sequences are exact
\par
{\bf -i)}   
\begin{gather}   %%%%   Eqn  (112)  Begin
K_{0}^{loc} (\Lambda)   \overset{ (i_{1 \ast}, i_{2 \ast}) }{\longrightarrow} 
K_{0}^{loc} (\Lambda_{1})  \oplus K_{0}^{loc} (\Lambda_{2}) 
\overset{ j_{1 \ast} - j_{2 \ast} }{\longrightarrow}   K_{0}^{loc}(\Lambda^{'})
\end{gather}  %%%%   Eqn  (112)   End
\par
{\bf -ii)}
\begin{gather}    %%%%   Eqn  (113)  Begin
K_{1}^{loc} (\Lambda) \otimes \mathbb{Z}[\frac{1}{2}]  \overset{(i_{1 \ast}, i_{2 \ast})}{\longrightarrow}
K_{1}^{loc} (\Lambda_{1}) \otimes \mathbb{Z}[\frac{1}{2}]   \oplus K_{1}^{loc} (\Lambda_{2}) \otimes \mathbb{Z}[\frac{1}{2}]   \overset{ j_{1 \ast} - j_{2 \ast} }{\longrightarrow} 
K_{1}^{loc} (\Lambda^{'})  \otimes \mathbb{Z}[\frac{1}{2}]  
\end{gather}   %%%%   Eqn  (113)   End
\par
{\bf -iii)}
\begin{gather}  %%%%   Eqn  (114)  Begin
K_{1}^{loc} (\Lambda_{1})  \oplus K_{1}^{loc} (\Lambda_{2})  \overset{ j_{1 \ast} - j_{2 \ast} }{\longrightarrow} 
K_{1}^{loc} (\Lambda')   \overset{\partial}{\longrightarrow} 
K_{0}^{loc} (\Lambda)   \overset{ (i_{1 \ast}, i_{2 \ast}) }{\longrightarrow} 
K_{0}^{loc} (\Lambda_{1})  \oplus K_{0}^{loc} (\Lambda_{2}). 
\end{gather} %%%%   Eqn  (114)   End
\end{theorem} %%%%%%%  Theorem  51  - End
%%%%%%%%%%%%%%%%%%%%%%%%%%%%%%%%%%%%%%%%%%%
\begin{proof}   %%%%%   Proof  Theorem  51  Begin
%%%%%%%%%%%%%   BEGIN  PROOF  PARTS
\par
%%%%%%%%%%%%%    BEGIN  PROOF  PART -i)

{\bf -i)}  It is easy to verify that
$Im (i_{1 \ast}, i_{2 \ast}) \subset Ker (j_{1 \ast} - j_{2 \ast})$.
\par
We verify that $Ker (j_{1 \ast} - j_{1 \ast}) \subset  Im  (i_{1 \ast}, i_{2 \ast})$. 
\par
Let $[p_{1}] - [p_{2}] \; , \;  [q_{1}] - [q_{2}] \in K_{0}^{loc} (\Lambda_{1, \mu}) \oplus K_{0}^{loc} (\Lambda_{1, \mu}) $
be such that 
$$
0 = j_{1, \ast}([p_{1}] - [p_{2}]) - j_{2, \ast} ([q_{1}] - [q_{2}]), 
$$
where $p_{1}, p_{2}, q_{1}, q_{2}$ are idempotents. The pair of $K$-theory classes may be re-written
$$
0 = j_{1, \ast}([p_{1} + \bar{p}_{2}] - [p_{2}+ \bar{p}_{2})] - j_{2, \ast} ([q_{1}+ \bar{q}_{2}] - [q_{2}+ \bar{q}_{2}]) 
$$
or
$$
0 = j_{1, \ast}([p_{1} + \bar{p}_{2}] - [1_{r}] - j_{2, \ast} ([q_{1}+ \bar{q}_{2}] - [1_{s}]). 
$$
By adding all sides a trivial idempotent of sufficiently large size, we may assume that $r = s$ is large. This relation may be re-written
$$
0 = j_{1, \ast}([p_{1} + \bar{p}_{2}] - j_{2, \ast} ([q_{1}+ \bar{q}_{2}] ). 
$$
This means that there exists an idempotent $\xi \in \mathbb{I}demp_{N}(\Lambda^{'}_{\mu})$  such that the idempotents
$$
( j_{1, \ast}([p_{1} + \bar{p}_{2}] + \xi) - ( j_{2, \ast} ([q_{1}+ \bar{q}_{2}] + \xi) 
$$
are isomorphic. We add further the idempotent $\bar{\xi} := 1_{N} - \xi$ to get isomorphic idempotents

$$
( j_{1, \ast}([p_{1} + \bar{p}_{2} ] + \xi + \bar{\xi}),   \;\;  ( j_{2, \ast} ([q_{1}+ \bar{q}_{2}] + \xi + \bar{\xi}) \in 
\mathbb{I}demp_{\bar{N}} (\Lambda^{'}_{\mu})
$$
( $\bar{N}$ being the size of these idempotents)
or
$$
( j_{1, \ast}([p_{1} + \bar{p}_{2} ] + 1_{N}), \;\;  ( j_{2, \ast} ([q_{1}+ \bar{q}_{2}] + 1_{N}). 
$$
This means there exists $u \in \mathbb{GL}_{\bar{N}}(\Lambda^{'})$ which conjugates these two idempotents.
\par
Theorem 34 says that there exits the idempotent $p \in \mathbb{I}demp_{2 \bar{N}}$ such that
$$
i_{1 \ast} p = (p_{1} + \bar{p}_{2} + 1_{N}) \oplus 1_{\bar{N}}
$$
and 
$$
i_{2 \ast} p = U \; ( (q_{1} + \bar{q}_{2} + 1_{N}) \oplus 1_{\bar{N}} ) \; U^{1}
$$
where 
$$
U \in \mathbb{GL}_{2 \bar{N}}(\Lambda^{'}_{\mu}).
$$
This means the image of the $K$-theory class of $p - 1_{\bar{N} }$ through the pair of homomorphisms $(i_{1, \ast}, i_{2, \ast})$ is
$$
(i_{1, \ast}, i_{2, \ast}) [p - 1_{\bar{N}}] =   
$$
$$
= (\;  (p_{1} + \bar{p}_{2} + 1_{N}) \oplus 1_{\bar{N}} - 1_{\bar{N}} ,  
U \; ( (q_{1} + \bar{q}_{2} + 1_{N}) \oplus 1_{\bar{N}} ) \; U^{1} - 1_{\bar{N}}
\; ) = 
$$

$$
= (\;  (p_{1} + \bar{p}_{2} + 1_{N}) \oplus 1_{\bar{N}} - 1_{\bar{N}} ,  
U \; ( (q_{1} + \bar{q}_{2} + 1_{N}) \oplus 1_{\bar{N}}  - 1_{\bar{N}}  ) \;  \; U^{1} 
\;  = 
$$
$$
= ([p_{1} - p_{2}], [q_{1} - q_{2}],
$$
which completes the proof of the assertion.

\par
%%%%%%%%%%%%%    BEGIN  PROOF  PART -ii)

{\bf -ii)}  The only problem might occur in verifying that $Ker (j_{1 \ast} - j_{2 \ast}) \subset  Im  (i_{1 \ast}, i_{2 \ast}).$ The proof of it is provided by Theorem 36.  
The remaining verifications are obvious. 
\par
{\bf -iii)} 
%%%%%%%%%%%%%%%%%%   BEGIN  PROOF PART  -iii)
\par
{\bf -iii.1)}  
Let $u \in \mathbb{GL}(\mathit{A}_{\mu})$.
Lemma 38 says that $\mathcal{O}(u) \otimes \mathbb{Z}[\frac{1}{2}]$ lifts as an invertible element over $\Lambda_{2, \mu}$.  Remark 48 completes the argument stating that $\partial u = 0 \in K_{0}(\Lambda_{\mu} \otimes \mathbb{Z}[\frac{1}{2}]$. This result takes care 
of the elements belonging to $\mathbb{O}(\Lambda^{'}_{\mu})$.

\par
We verify that  $\partial \circ  j_{1 \ast} = 0$. Indeed, 
let $ \tilde{U} \in \mathbb{GL}_{n}(\Lambda_{1, \mu})$ and  $U = j_{1 \ast} \tilde{U}$. Then $\partial \circ j_{1, \ast} [\tilde{U}] =
\partial [U]$.  The elements $U$, $U^{-1}$ have the liftings $A = \tilde{U}$ and $B = \tilde{U}^{-1}$. Then the corresponding elements
$S_{0}$  and $S_{1}$ are equal to zero. Therefore  
\begin{gather}    %%%%   Eqn  (115)   Begin 
\partial [U] =  
[\begin{pmatrix}
(0,0) &  (0,0) \\
(0,0) &  (1, 1)
\end{pmatrix}]
-
[\begin{pmatrix}
(0,0) &  (0,0) \\
(0,0) &  (1, 1)
\end{pmatrix}] =
0.
\end{gather}   %%%%%   Eqn   (115)   End
%%%%%%%%%%%%%%%%%%%%%
\par
{\bf -iii.2)} $   (i_{1 \ast}, i_{2 \ast})  \circ \partial = 0$.
Let $[U] \in K_{1}^{loc}(\Lambda')$ and let $\partial [U] =  [\mathbb{P}_{U}]  - [(e_{2}, e_{2})]$, where $\mathbb{P}_{U}$ is given by
formula (89). Then
\begin{equation}   %%%%%   Eqn  (116)   Begin
i_{1, \ast} \partial [U]  =  [L e_{1} L^{-1}] - [ e_{2}] = [e_{1}] - [e_{2}] = 0. 
\end{equation}  %%%%%   Eqn  (116)   End
On the other side,
\begin{equation}   %%%%%   Eqn  (117)   Begin
i_{2, \ast} \partial [U]  =  [e_{2}] - [ e_{2}] = 0. 
\end{equation}   %%%%%   Eqn   (117)   End
\par

%%%%%%%%%%%%%%%%%%%%%%%%%%%%%%
{\bf -iii.3)} Suppose $[u] \in K_{1}^{loc}(\Lambda^{'})$ and $\partial [u] = 0$. We have to prove that 
there exist invertible matrices $U_{1} \in \mathbb{GL}(\Lambda_{1, \mu})$ and 
$U_{1} \in \mathbb{GL}(\Lambda_{1, \mu})$
such that 
\begin{equation}   %%%%%   Eqn  (118)   Begin
[u] = (j_{1, \ast} - j_{2, \ast}) 
(U_{1}, U_{2}) =  j_{1, \ast}  (U_{1})  - j_{2, \ast} (U_{2}) 
\end{equation}   %%%%%   Eqn  (118)   End
%%%%%%%%%
\par
$\partial [u] = 0$ means (after stabilisation) that $\mathbb{P}_{u}$ is l-conjugate to $[ ( e_{2}, e_{2} ) ] $, i.e. there exists
%%%%%%%%%%
\begin{gather}   %%%%%   Eqn  (119)   Begin
U = (U_{1}, U_{2})  \in \mathbb{GL}_{n}(\Lambda_{1, \mu}) \otimes \mathbb{GL}_{n}(\Lambda_{2, \mu}) ,
\end{gather}
which, as components of an invertible matrix over $\Lambda_{\mu}$, satisfy the compatibility condition
\begin{equation}  %%%%%   Eqn (120)
j_{1} (U_{1})  = j_{2} (U_{2})
\end{equation}
 %%%%%   Eqn  (120)   End
 such that
  %%%%%%%%
\begin{gather}  %%%%%   Eqn  (121)   Begin
\mathbb{P}_{U} :=    
\begin{pmatrix}
(S_{0}^{2}, \; 0)  &   (S_{0} (1 + S_{0}) B, \; 0)  \\
(S_{1}A, \; 0)   &  (1 - S_{1}^{2}, \; 1 )
\end{pmatrix}
= U .(e_{2}, e_{2}).  U^{-1}.
\end{gather}  %%%%%   Eqn  (121)   End
%%%%%%%
In other words, 
\begin{equation}  %%%%%   Eqn   (122)  Begin
\partial u =   [ U . (e_{2}, e_{2}).  U^{-1}] - [ ( e_{2}, e_{2} ) ] = 
 [ U_{1} .e_{2} .U_{1}^{-1}  \;,  \; U_{2} e_{2} U_{2}^{-1}  )]  - [ ( e_{2}, e_{2} ) ].
\end{equation}   %%%%%   Eqn  (122)  End
We assume $\Lambda_{1, \mu} = \Lambda_{2, \mu} $.
The class $  [ U .(e_{2}, e_{2}). U^{-1} ] \in \mathbb{I}demp(\Lambda_{\mu})$ remains unchanged by further conjugating it by the double invertible matrix
\begin{equation}   %%%%%   Eqn  (123)  Begin 
\tilde{U} = (U_{1}^{-1}, U_{1}^{-1}) \in \mathbb{GL} (\Lambda_{\mu}).
\end{equation}  %%%%%   Eqn  (123)  End
Denote by $Inn (A)$ the inner auto-mporhism associated with the invertible element $A$. Noting that the inner auto-morphisms satisfy
\begin{equation}   %%%%   Eqn  (124)  Begin
Inn (A) \circ Inn (B) = Inn (A.B),
\end{equation} %%%%%%   Eqn  (124)  End
we have, after recalling Theorem 27. -i)
\begin{gather*}
\partial u = 
 [ U_{1} .e_{2} .U_{1}^{-1}  \;,  \; U_{2} e_{2} U_{2}^{-1}  )]  - [ ( e_{2}, e_{2} ) ] = \\
 = [ \tilde{U} (U_{1} .e_{2} .U_{1}^{-1}  \;,  \; U_{2} e_{2} U_{2}^{-1}  ) \tilde{U}^{-1}]  - [ ( e_{2}, e_{2} ) ] =
 \end{gather*}
\begin{equation}   %%%%%%   Eqn  (125)   Begin
 = [ U_{1}^{-1 }(U_{1} .e_{2} .U_{1}^{-1} U_{1} \;,  \; U_{1}^{-1 }U_{2} e_{2} U_{2}^{-1} U_{1} ) \tilde{U}^{-1}]  - [ ( e_{2}, e_{2} ) ] =
 p(1_{N}, 1_{N},  U_{1}^{-1} U_{2}).
\end{equation}   %%%%%%   Eqn  (125)   End
Note that $ U_{1}^{-1} U_{2} \in \mathbb{GL}_{N}(\Lambda_{1, \mu})= \mathbb{GL}_{N}(\Lambda_{2, \mu})$. This relation proves the assertion -iii.3).

%%%%%%%%%%%%%%%.%%%%%%%%%%%%%%%%%%%%%%%%%%%%%%%%%%%%%%%%%%%%%%%%%%

-{\bf iii.4})  Let $[p]  = ([p_{1}] - [p_{2}])  \in K_{0} (\Lambda_{\mu}) $ ($p_{1} \in \mathit{Idemp}_{m}(\Lambda_{\mu})$ and $p_{2} \in \mathit{Idemp}_{n}(\Lambda_{\mu})$)
and suppose $(i_{1 \ast}, i_{2 \ast}) [p]  = 0$. 
We have to show that $[p_{1}] - [p_{2}] = \partial [U]$, with $ U \in \mathbb{GL} (\Lambda'_{\mu})$.
\par
We may assume that the idempotent $p_{2}$ is trivial; indeed, the matrix
$(1- p_{2}, 1 - p_{2} )$, seen as a matrix in $\mathbb{M}_{n} (\Lambda_{\mu})$, is an idempotent over $\Lambda_{\mu}$ 
 (here we have assumed that $\Lambda_{1, \mu} = \Lambda_{2, \mu}$) and therefore
 $[p_{1}] - [p_{2}] = [p_{1} \oplus (1 - p_{2}) ] - [p_{2} \oplus ( 1 - p_{2})]$. 
 Here 
 \begin{gather}   %%%%%%   Eqn   (126)-(127)  Begin
 \begin{align}
 p_{1} \oplus (1 - p_{2}) &:= \bar{p}_{1} \in \mathit{Idemp}_{m+n} (\Lambda_{\mu})  \; \text{and}\\
 p_{2} \oplus ( 1 - p_{2}) &:= \bar{p}_{2} \in Idemp_{2n}(\Lambda_{\mu}) \text{ is  a trivial idempotent.}
 \end{align}
 \end{gather}   %%%%   Eqn   (126)-(127)   End
%%%%%%%%%%%%%%%%%%%
 \par
 Denote   
 \begin{gather}   %%%%%   Eqn (128)  Begin
 e_{(k,l)} :=
 \begin{pmatrix}
 0 & 0 \\
 0 & 1_{l}
 \end{pmatrix}
 \in \mathbb{M}_{k+l} (\mathbb{C}).
 \end{gather}   %%%%%   Eqn (128)  End
 \par
 Let
 \begin{gather}    %%%%%   Eqn (129-130)  Begin
 \begin{align}
(i_{1 \ast}, i_{2\ast})\;  \bar{p}_{1} & := (\bar{p}_{11}, \bar{p}_{12}),   \;\;  \bar{p}_{1k} \in \mathbb{M}_{m+n}(\Lambda_{k, \mu}), \hspace{0.2cm} k = 1, 2,
   with \;\; j_{1, \ast} \bar{p}_{11}  - j_{2, \ast} \bar{p}_{12} = 0, \\
(i_{1\ast}, i_{2\ast})  \; \bar{p}_{2} &:= (e_{(0,2n)}, e_{(0,2n)}),  \;\;  e_{(0,2n)} \in \mathbb{M}_{2n}(\Lambda_{k, \mu}). 
\end{align}
\end{gather}   %%%%%   Eqn (129-130)  End

\par
By hypothesis we know also that
\begin{gather}  %%%%%   Eqn  (131)-(132)   Begin
 [ \bar{p}_{11}] - [e_{(0,2n)}] = 0 \in K_{0}(\Lambda_{1, \mu}) \\  %%%%%%  Eqn  (131)
 [ \bar{p}_{12}] - [e_{(0,2n)}]  = 0 \in K_{0}(\Lambda_{2, \mu}).   %%%%%%  Eqn  (132)
\end{gather}   %%%%%  Eqn   (131)-(132)    End
Equations (131),  (132) tell that after possible stabilisations (let $\tilde{p}_{11} \in \mathbb{I}demp_{2n+q} (\Lambda_{1, \mu})$, resp.
$\tilde{p}_{12} \in \mathbb{I}demp_{2n+q} (\Lambda_{2, \mu})$ and $e_{(q, 2n)}$ be the stabilised idempotents (with the same $q$, sufficiently large), there exist invertible elements
$u_{1} \in \mathbb{GL}_{q+2n}(\Lambda_{1, \mu} )$,      $u_{2} \in \mathbb{GL}_{q+2n}(\Lambda_{2, \mu})$,      such that

\begin{gather}  %%%%%  Eqn  (133)-(134)  Begin
\tilde{p}_{11} = u_{1} \; \tilde{p}_{21} \; u_{1}^{-1} = u_{1} \; e_{(q,2n)} \; u_{1}^{-1} \\
\tilde{p}_{12} = u_{2} \; \tilde{p}_{22} \; u_{2}^{-1} = u_{2} \; e_{(q,2n)} \;  u_{2}^{-1}.
\end{gather}   %%%%%   Eqn  (133)-(134)   End 
After such stabilisations, the new double matrix  $\tilde{p}_{1} := (\tilde{p}_{11} , \tilde{p}_{12} )$ is still an idempotent in
$\mathbb{M}_{q+2n}(\Lambda_{\mu})$. 
\par
The isomorphisms $u_{1}, u_{2}$ tell that the pair  $\tilde{p}_{1} := (\tilde{p}_{11} , \tilde{p}_{12} )$ is $(\Lambda_{1, \mu}, \Lambda_{2, \mu})$-isomorphic to the pair of trivial idempotents $(e_{(q,2n)}, e_{(q,2n)})$. Notice, however, that in general,  $j_{1 \ast} u_{1} \neq j_{2 \ast} u_{2}$; this is why this isomorphism is not necessarily a  $\Lambda_{ \mu}$-isomorphism.
%%%%%%%%%%%%%%%%%%%%%%%%%%%%%%%%%%%%%%%%%%%%%%%%%%%%%%%%%%%%%%%%
%%%%%%%%%%%%%%%%%%%%%%%%%%%%%%%%%% 
\par
We may further simplify the description of the idempotent $\tilde{p}_{1} \in \mathbb{I}demp(\Lambda_{\mu})$. 
To do this we replace the idempotent $\tilde{p}_{1}$ by the idempotent $\tilde{\tilde{p}}_{1} :=  (\tilde{\tilde{p}}_{11} , \tilde{\tilde{p}}_{12} ) :=
  U \tilde{p}_{1} U^{-1} $, where $U$ is the double matrix 
\begin{equation}      %%%%   Eqn   (135)    Begin
\tilde{U} := (u_{2}^{-1}, u_{2}^{-1}) \in \mathbb{GL}_{q+2n} (\Lambda_{\mu}).
\end{equation}   %%%%   Eqn   (135)   End    
After this modification  
%%%%%%%%%%%%%%%%%%%%%%%%%%%%%%%%%% 
\begin{gather}  %%%%%  Eqn  (136-137)  Begin
\begin{align}
\tilde{\tilde{p}}_{11} =&   u_{2}^{-1} \; \tilde{p}_{11} \; u_{2} =   (u_{2}^{-1} \; u_{1}) \; e_{(q,2n)}
 \; (u_{2}^{-1} \; u_{1})^{-1} \\
\tilde{\tilde{p}}_{12} =& u_{2}^{-1} \;  \tilde{p}_{12} \; u_{2} = u_{2}^{-1}    (u_{2} \; e_{(q,2n)}
\; u_{2}^{-1})      u_{2} = e_{(q,2n)}.
\end{align}
\end{gather} %%%%%   Eqn  (136-137)   End 
\par
As a result of all these transformations, the element $[p]$ is stably $\Lambda_{\mu}$-isomorphic to 
\begin{equation}   %%%%%   Eqn   (138)   Begin
[p] = [p_{1}] - [p_{2}] = [( u_{2}^{-1} \; u_{1}) \; e_{(q,2n)} \; (u_{2}^{-1} \; u_{1})^{-1}, \; e_{(q,2n)}] - [e_{(q,2n)}, \; e_{(q,2n)}].
\end{equation}   %%%%%%   Eqn  (138)   End
%%%%%%%%%%%%%%%%%%%%%%%%%%%%%%%%%%%%%%%%%%%%%%%%%%%%%%
From the equations (133), (134) we get

\begin{gather}  %%%%%  Eqn  (139)  Begin
j_{1, \ast } {\tilde{\tilde{p}}}_{11} =   j_{1, \ast } (\;  (u_{2}^{-1} \; u_{1}) \; e_{(q,2n)} \; (u_{2}^{-1} \; u_{1})^{-1} \;) = 
j_{2, \ast } {\tilde{\tilde{p}}}_{12}.
\end{gather} %%%%%   Eqn  (139)   End 

\par
This shows that 
\begin{equation}  %%%%%   Eqn  (140)   Begin
\phi := j_{1 \ast} (u_{2}^{-1} \; u_{1}) \in \mathbb{GL}(\Lambda'_{1, \mu})
\end{equation}  %%%%%%   Eqn  (140)   End
establishes and isomorphism 
between $j_{1 \ast}\tilde{\tilde{p}}_{11}$ and $j_{1 \ast}e_{(q,2n)} = e_{(q,2n)}$.

\par     %%%%%%%%    
%%%%%%%%%%%%%%%%%%%%%%%%%%%%%%%
Recall (see \S 10.2) that the construction of $\partial (U)$ involves the lifting $A, \; resp. \; B$, of $U$, resp.  $U^{-1}$, in 
$\mathbb{M}_{2n}(\Lambda_{1, \mu})$. We look for $A, B \in \mathbb{M}_{2n}(\Lambda_{1, \mu})$ such that
$j_{1, \ast} A = U$  and $j_{1, \ast} B = U^{-1}$.
\par
We may choose the liftings
\begin{gather}   %%%%%   Eqn  (141)-(142)  Begin
A = u _{1} \; u_{2}^{-1} \in \mathbb{GL}_{q+2n}(\Lambda_{1, \mu}) \\
B =  u _{2} \; u_{1}^{-1} \in \mathbb{GL}_{q+2n}(\Lambda_{1, \mu}).
\end{gather}  %%%%   Eqn   (141)-(142)   End
From this we get $S_{0}= 0$ and  $S_{1}= 0$.  
\par
%%%%%%%%%%%%%%%%%%%%%%%%%%%%%%%%%%%%%%%%%%%%%%%%%%%%%%%%%%%%%
We get further
\begin{gather}  %%%%%   Eqn  (143)  Begin
\mathit{P}
= 
\begin{pmatrix}
S_{0}^{2} & S_{0}(1 + S_{0})B \\
 S_{1} A & 1 - S_{1}^{2}
\end{pmatrix} =
\begin{pmatrix}
0& 0 \\
0 & 1 
\end{pmatrix}
\in \mathbb{M}_{2(m+n+q)}(\Lambda_{\mu}).
\end{gather}   %%%%%   Eqn  (143)   End
This shows that   
\begin{equation}
\partial U = [\tilde{p}_{11}] -  [\tilde{p}_{22}] = 
( [\tilde{p}_{11}] - [e_{0,2n}] ) 
 - ( [\tilde{p}_{22}]  - [e_{0,2n}] )  = [p_{1}] - [p_{2}].
\end{equation}
These complete the proof of Theorem 51.
\end{proof}   %%%%%   Proof  Theorem 51 -iii)  -  End
%%%%%%%%%%%%%    

%%%%%%%%%%%%%%%%%%%%%%%%%%%%%%%%%%%%%%


\newpage
\begin{thebibliography}{99}

\bibitem{Atiyah_Singer_I} Atiyah M., Singer I. M.: The Index of elliptic operators: I,  Ann. of Math.  87  (1968),     484 - 530 

\bibitem{Atiyah_Singer_IV} Atiyah M., Singer I. M.: The Index of elliptic operators: IV,  Ann. of Math.  92  (1970),     119 - 138 

\bibitem{Bass}  Bass H.: $K$-theory and stable algebra.  Publ. Math. I.H.E.S. 22 (1964), pp. 5 - 60.


\bibitem{Blackadar} Blackadar B.: $K$-Theory for Operator Algebras, Second Ed, 1998, Cambridge University Press.

\bibitem{Connes} Connes A.: Connes A.: Noncommutative Geometry, Academic Press, 1994.

\bibitem{Connes_Moscovici} Connes A., Moscovici H.: Cyclic Cohomology, the Novikov Conjecture Vol. 29 and Hyperbolic Groups,  Topology  Vol. 29, pp.345-388, 1990.

\bibitem{Connes_Sullivan_Teleman} Connes A., Sullivan D., Teleman N.: Quasiconformal Mappings, Operators on Hilbert Space and
Local Formulae for Characteristic Classes, Topology Vol. 33, Nr. 4, pp. 663-681, 1994.

\bibitem{Cuntz} Cuntz J.: Cyclic Theory, Bivariant $K$-theory and the Bivariant Chern-Connes Character. Encyclopedia of Mathematical Sciences. Operator Algebras and Non-commutative Geometry. J. Cuntz, V. F. R. Jones Eds., Springer - Verlag, Berlin, 2004.

\bibitem{Donaldson_Sullivan} Donaldson S. K., Sullivan D.: Quasi-conformal 4-Manifolds, Acta Mathematica., Vol. 163 (1989), pp. 181Ð252.
 
 \bibitem{Hirzebruch_I} Hirzebruch F.:  Neue topologische Methoden in der algebraischen Geometrie. Berlin, Springer, 1956.
 
 \bibitem{Hirzebruch_II} Hirzebruch F.: Elliptische Differentialoperatoren auf Mannigfaltigkeiten, Arbaitsgemeinschaft fur Forschung
 des Landes Nordrhein-Westfalen, 33 (1965), 583 - 608.
 
 \bibitem{Karoubi} Karoubi M.: K-THEORY.  An elementary introduction. Conference at the Clay Mathematics Research Academy.
 
 \bibitem{Kasparov} Kasparov G. G.: Topological invariants of elliptic operators, I: $K$-homology. Math. USSR Izvestija, Vol. 9 (1975), 751 - 792
 
 \bibitem{Milnor} Milnor J.:  Introduction to Algebraic $K$-Theory, Annals of Mathematics Studies Nr. 72, Princeton, 1971.
 
 \bibitem{Puschnigg} Puschnigg M.: Diffeotopy Functors of Ind-Algebras and Local Cyclic Cohomology. Documenta Mathematica, Vol. 8, pp. 143-245, 2003.
 
 \bibitem{Rordam_Lansen_Lautsen} Rordam M., Lansen F., Lautsen N.: An Introduction to $K$-Theory for $C^{\ast}$-Algebras.
 London Math. Soc. Student Texts Nr. 49, Cambridge, 2000
 
 \bibitem{Rosenberg} Rosenberg J.: Algebraic $K$-Theory and its Applications. Graduate Texts Nr. 147, Springer, Berlin, 
 1994.
 
 \bibitem{Spanier} Spanier E. H.: Algebraic Topology, McGraw - Hill Series in Higher Mathematics, New York, 1966.
 
 \bibitem{Swan} Swan R.: Algebraic $K$-Theory. Lecture Notes in Mathematics, Vol. 76,  Springer,  1968.

\bibitem{Teleman_IHES} Teleman N.: The Index of Signature Operators on Lipschitz Manifolds.  
Publ. Math. Paris, IHES, Vol. 58, pp. 251-290, 1983 

\bibitem{Teleman_Acta} Teleman N.: The Index Theorem on Topological Manifolds. Acta Mathematica  ???

\bibitem{Teleman_arXiv_I} Teleman N.: $Local^{3}$ Index Theorem.  arXiv: 1109.6095v1 [math.KT],   28 Sep. 2011.

\bibitem{Teleman_arXiv_II} Teleman N.: {\em Local} Hochschild Homology of Hilbert-Schmidt Operators on Simplicial Spaces. 
             arXiv  hal-00707040, Version 1,  11 June 2012.
             
\bibitem{Teleman_arXiv_III}  Teleman N.: The {\em Local} Index Theorem.     HAL-00825083,   arXiv :1305.5329v2 [math.KT] ,     
                24 May 2013.       
                
\bibitem{Teleman_arXiv_IV}  Teleman N.:    The Local Index Theorem.  23/mag/2013 - arXiv.  Math.  arXiv:1305.5329.                   
                                                                                                               
 \bibitem{Weibel}  Weibel C. A.: Algebraic $K$-Theory, 2012.    
 
\bibitem{Whitehead}  Whitehead  J. H. C.: Simple Homotopy Types. Amer. J. Math. (72), 1 - 57, 1950.                                                                                                         
                                                                                                               
\bibitem{Wegge-Olsen} Wegge-Olsen N. E.: $K$-Theory and $C^{\ast}$-Algebras. A Friendly Approach. Oxford University Press, 1994.                                                                                                            


\end{thebibliography}
 \end{document}